\documentclass[12pt, reqno]{amsart}
\makeatletter
\DeclareMathAlphabet{\mathpzc}{OT1}{pzc}{m}{it}

\usepackage[margin=2cm]{geometry}
\usepackage{amsmath, amssymb, amsfonts, amstext, verbatim, amsthm, mathrsfs}
\usepackage{microtype}
\usepackage[all,arc,knot,poly]{xy}
\usepackage{mathtools}
\usepackage{aliascnt} 
\usepackage[colorlinks=true,linkcolor=blue,citecolor=blue,urlcolor=blue,citebordercolor={0 0 1},urlbordercolor={0 0 1},linkbordercolor={0 0 1}]{hyperref} 
\usepackage{enumerate}
\usepackage{stmaryrd}

\theoremstyle{plain}

\newcommand{\refnewtheoremn}[4]{
\newaliascnt{#1}{#2}
\newtheorem{#1}[#1]{#3}
\aliascntresetthe{#1}
\expandafter\providecommand\csname #1autorefname\endcsname{#4}}

\newcommand{\refnewtheorem}[3]{\refnewtheoremn{#1}{#2}{#3}{#3}}

\newtheorem{theorem}{Theorem}[section]
\refnewtheorem{lemma}{thm}{Lemma}
\refnewtheorem{claim}{thm}{Claim}
\refnewtheorem{corollary}{thm}{Corollary}
\refnewtheorem{conj}{thm}{Conjecture}
\refnewtheorem{prop}{thm}{Proposition}
\refnewtheorem{proposition}{thm}{Proposition}
\refnewtheorem{quest}{thm}{Question}
\refnewtheorem{assumption}{thm}{Assumption}
\refnewtheorem{problem}{thm}{Problem}

\theoremstyle{definition}
\refnewtheorem{remark}{thm}{Remark}
\refnewtheorem{remarks}{thm}{Remarks}
\refnewtheorem{defn}{thm}{Definition}
\refnewtheorem{definition}{thm}{Definition}
\refnewtheorem{notn}{thm}{Notation}
\refnewtheorem{const}{thm}{Construction}
\refnewtheorem{example}{thm}{Example}
\refnewtheorem{examples}{thm}{Examples}
\refnewtheorem{fact}{thm}{Fact}

\DeclareMathOperator{\Crit}{\operatorname{Crit}}
\DeclareMathOperator{\Res}{\operatorname{Res}}
\DeclareMathOperator{\ind}{\operatorname{ind}}
\DeclareMathOperator{\Hom}{\operatorname{Hom}}
\DeclareMathOperator{\PGL}{\operatorname{PGL}}
\DeclareMathOperator{\Sch}{\operatorname{Sch}}
\DeclareMathOperator{\Spec}{\operatorname{Spec}}
\DeclareMathOperator{\Sym}{\operatorname{Sym}}
\DeclareMathOperator{\Aut}{\operatorname{Aut}}

\renewcommand{\Im}{\operatorname{Im}}

\begin{document}

\title[On the wall-crossing formula for quadratic differentials]{On the wall-crossing formula for \protect\\quadratic differentials}
\author{Dylan G.L. Allegretti}

\date{}

\maketitle

\begin{abstract}
We prove an analytic version of the Kontsevich-Soibelman wall-crossing formula describing how the number of finite-length trajectories of a quadratic differential jumps as the differential is varied. We characterize certain birational automorphisms of an algebraic torus appearing in this wall-crossing formula using Fock-Goncharov coordinates. As an application, we compute the Stokes automorphisms for the Voros~symbols appearing in the exact WKB~analysis of Schr\"odinger's~equation.
\end{abstract}

\section{Introduction}

The concept of a quadratic differential is fundamental in several areas of low-dimensional geometry and dynamics. Given a Riemann surface~$S$, a holomorphic quadratic differential $\phi$ on~$S$ is defined as a holomorphic section of $\omega_S^{\otimes2}$ where $\omega_S$ is the holomorphic cotangent bundle of~$S$. Such a differential~$\phi$ determines a flat metric on~$S$ with singularities at the zeros of~$\phi$. Of particular importance in the theory are certain geodesics on~$S$ with respect to the flat metric. These geodesics are the so called finite-length trajectories. Over the past four decades, the problem of counting finite-length trajectories on a surface equipped with a quadratic differential has led to many remarkable results \cite{Masur1, Masur2, Veech, EskinMasur, EskinMasurZorich, EMM}.

In a seemingly unrelated development, Kontsevich and Soibelman described an approach to Donaldson-Thomas theory that can be used to associate numerical invariants to moduli spaces of objects in a 3-Calabi-Yau triangulated category~\cite{KontsevichSoibelman1,KontsevichSoibelman2,KontsevichSoibelman3}. These invariants are known as BPS~invariants since they provide a mathematical approach to counting BPS~states in physics. The~BPS~invariants depend on some extra data, namely a choice of Bridgeland stability condition~\cite{Bridgeland07}. The set of all stability conditions has the structure of a complex manifold, and the BPS~invariants remain constant as one varies the choice of stability condition within this manifold, except at certain real codimension one ``walls'' where they can change discontinuously. The celebrated Kontsevich-Soibelman wall-crossing formula describes precisely how the BPS~invariants change as one crosses a wall in the space of stability conditions~\cite{KontsevichSoibelman1}.

As part of their work on the physics of four-dimensional supersymmetric quantum field theories, Gaiotto, Moore, and Neitzke proposed a relationship between Donaldson-Thomas theory and the theory of quadratic differentials~\cite{GMN}. Their proposal, which was later proved mathematically by Bridgeland and Smith~\cite{BridgelandSmith}, implies that in some cases the BPS invariants defined by Kontsevich and Soibelman count finite-length trajectories of meromorphic quadratic differentials. In particular, the wall-crossing formula encodes how these integers counting finite-length trajectories change as one varies the quadratic differential.

The purpose of the present paper is to study the Kontsevich-Soibelman wall-crossing formula in the context of quadratic differentials. Although the wall-crossing formula is usually formulated as an identity between infinite products of automorphisms of a formal power series ring, we will show using Fock-Goncharov coordinates~\cite{FockGoncharov1} that in our context these products can be viewed as birational transformations of a complex algebraic torus. As an application of our main result, we compute the Stokes automorphisms for the Voros symbols appearing in the exact WKB~analysis of Schr\"odinger's equation on a compact Riemann surface~\cite{Allegretti19a,IwakiNakanishi}.

\subsection{Basic setup}

In order to apply results from Donaldson-Thomas theory, we will restrict attention in this paper to a class of quadratic differentials studied by Gaiotto, Moore, and~Neitzke~\cite{GMN} and known as GMN~differentials. Given a compact Riemann surface~$S$, we will define a GMN~differential on~$S$ to be a meromorphic section of~$\omega_S^{\otimes2}$ satisfying some genericity assumptions which determine the possible orders of its critical points.

Let $\phi$ be a GMN differential on a compact Riemann surface~$S$. Near any point of~$S$ which is not a zero or pole of $\phi$, there is a local coordinate~$w$, well defined up to transformations $w\mapsto\pm w+\text{constant}$, such that $\phi=dw\otimes dw$. By pulling back the Euclidean metric on~$\mathbb{C}$ using these distinguished local coordinates, we get a flat metric on the complement of the zeros and poles. We will be interested in certain geodesics with respect to this flat metric. These geodesics are called finite-length trajectories and come in two types. The first is a saddle connection, which is roughly a geodesic connecting two zeros of~$\phi$. A saddle connection is closed if the zeros at its endpoints coincide. The second type of finite-length trajectory is a closed geodesic, which forms a loop in the complement of all zeros and poles. A closed geodesic is always contained in an annulus foliated by homotopic closed geodesics on~$S$. We refer to such an annulus as a cylinder for~$\phi$. A cylinder is said to be degenerate if one of its boundaries is a pole of order two of~$\phi$.

The differential $\phi$ determines a canonical branched double cover $\pi:\Sigma_\phi\rightarrow S$ on which the square root $\sqrt{\phi}$ is a well defined 1-form. Let $\Sigma_\phi^\circ$ denote the complement in~$\Sigma_\phi$ of the preimages of all poles of~$\phi$ of order~$>1$, and let $\Gamma_\phi$ be the set of $\gamma\in H_1(\Sigma_\phi^\circ,\mathbb{Z})$ such that $\tau(\gamma)=-\gamma$ where $\tau$ is the involution exchanging the two sheets of the double cover~$\Sigma_\phi$. Then $\Gamma_\phi$ is a lattice of finite rank, and we have a group homomorphism 
\[
Z_\phi:\Gamma_\phi\rightarrow\mathbb{C}, \quad Z_\phi(\gamma)=\int_\gamma\sqrt{\phi},
\]
called the period map.

We will see that any finite length trajectory $\alpha$ has a natural lift $\widehat{\alpha}\in\Gamma_\phi$. Moreover, any two closed geodesics in the same cylinder have the same lift. Thus there is a class in~$\Gamma_\phi$ associated to any saddle connection or cylinder for~$\phi$. We will define a notion of genericity for GMN~differentials, and if $\phi$ is a generic GMN differential and $\gamma\in\Gamma_\phi$ is any class, we will consider the associated integer 
\begin{align*}
\Omega_\phi(\gamma) &= |\{\text{non-closed saddle connections of class $\pm\gamma$}\}| \\
&\quad - 2\cdot|\{\text{nondegenerate cylinders of class $\pm\gamma$}\}|
\end{align*}
which ``counts'' the finite-length trajectories of~$\phi$. By the work of Bridgeland and Smith~\cite{BridgelandSmith}, this integer coincides with the BPS invariant defined in a more general context by Kontsevich and Soibelman~\cite{KontsevichSoibelman1}.

\subsection{The wall-crossing formula}

If $\phi$ is any GMN differential, then an active ray is defined to be a ray in~$\mathbb{C}^*$ of the form $\ell=\mathbb{R}_{>0}\cdot Z_\phi(\gamma)$ where $\gamma$ is the class of some finite-length trajectory. We will be interested in certain birational maps associated to the active rays in the case when the differential $\phi$ is generic.

To define these maps, we consider an object called the twisted torus, defined as the set 
\[
\mathbb{T}_-=\{g:\Gamma_\phi\rightarrow\mathbb{C}^*:g(\gamma_1+\gamma_2)=(-1)^{\langle\gamma_1,\gamma_2\rangle}g(\gamma_1)g(\gamma_2)\}
\]
where $\langle-,-\rangle$ is the intersection pairing on~$\Gamma_\phi$. We will see that the twisted torus $\mathbb{T}_-$ has the natural structure of an algebraic variety whose coordinate ring is spanned as a vector space by the functions $x_\gamma:\mathbb{T}_-\rightarrow\mathbb{C}^*$ given by $x_\gamma(g)=g(\gamma)$. If $\phi$ is a generic GMN differential and $\ell\subset\mathbb{C}^*$ is any ray emanating from the origin, then there is a birational automorphism $\mathbf{S}_\phi(\ell)$ of~$\mathbb{T}_-$ given on functions by 
\[
\mathbf{S}_\phi(\ell)^*(x_\beta)=x_\beta\cdot\prod_{Z_\phi(\gamma)\in\ell}(1-x_\gamma)^{\Omega_\phi(\gamma)\cdot\langle\beta,\gamma\rangle}.
\]
Note that this automorphism is the identity if~$\ell$ is non-active.

We will extend this construction and describe, for any convex sector $\Delta\subset\mathbb{C}^*$, a partially defined automorphism $\mathbf{S}_\phi(\Delta)$ of~$\mathbb{T}_-$. To do this, we define a notion of height for any ray in~$\mathbb{C}^*$, and we consider the composition 
\[
\mathbf{S}_{\phi,<H}(\Delta)=\mathbf{S}_\phi(\ell_1)\circ\mathbf{S}_\phi(\ell_2)\circ\dots\circ\mathbf{S}_\phi(\ell_k)
\]
of the maps defined above where $\ell_1,\ell_2,\dots,\ell_k\subset\Delta$ are the rays of height $<H$ taken in the clockwise order. A result of Bridgeland~\cite{Bridgeland19} says there is a nonempty analytic open subset of~$\mathbb{T}_-$ on which the pointwise limit 
\[
\mathbf{S}_\phi(\Delta)=\lim_{H\rightarrow\infty}\mathbf{S}_{\phi,<H}(\Delta)
\]
exists and is holomorphic. This limiting function $\mathbf{S}_\phi(\Delta)$ is called the BPS automorphism associated to the sector~$\Delta\subset\mathbb{C}^*$. Similar analytic maps have been studied by Kontsevich and Soibelman~\cite{KontsevichSoibelman4}, who described a general framework for the analytic study of wall-crossing formulas.

In our context, the wall-crossing formula is an identity that relates the BPS automorphisms~$\mathbf{S}_\phi(\Delta)$ for different quadratic differentials~$\phi$. To state this result, we need a suitable moduli space of GMN~differentials. Following the approach of~\cite{BridgelandSmith}, our moduli space will be labeled by the combinatorial data of a marked bordered surface. This is defined as a compact oriented surface~$\mathbb{S}$ together with a finite set~$\mathbb{M}$ of marked points on~$\mathbb{S}$ such that every boundary component of~$\mathbb{S}$ contains at least one marked point. As we review below, any GMN~differential determines an associated marked bordered surface. For a given~$(\mathbb{S},\mathbb{M})$, there is a moduli space $\mathscr{Q}^\pm(\mathbb{S},\mathbb{M})$ parametrizing GMN differentials whose associated marked bordered surface is~$(\mathbb{S},\mathbb{M})$.

The lattices $\Gamma_\phi$ define a local system over this moduli space $\mathscr{Q}^\pm(\mathbb{S},\mathbb{M})$. If $\phi_t$, $t\in[0,1]$, is a one-parameter family of GMN~differentials in $\mathscr{Q}^\pm(\mathbb{S},\mathbb{M})$ then we can use the flat connection of this local system to identify the $\Gamma_{\phi_t}$ to a single lattice. In particular, we can view the twisted tori associated with the differentials $\phi_t$ as a single algebraic variety. We then have the following analytic version of the Kontsevich-Soibelman wall-crossing formula.

\begin{theorem}
\label{thm:introWCF}
Let $(\mathbb{S},\mathbb{M})$ be a marked bordered surface which is not a closed surface with exactly one marked point, and let $\Delta\subset\mathbb{C}^*$ be a convex sector. Suppose $\phi_t$, $t\in[0,1]$, is a path in~$\mathscr{Q}^\pm(\mathbb{S},\mathbb{M})$ with general endpoints such that the boundary rays of~$\Delta$ are non-active for each differential~$\phi_t$. Then 
\[
\mathbf{S}_{\phi_0}(\Delta)=\mathbf{S}_{\phi_1}(\Delta).
\]
\end{theorem}

Note that the BPS invariants for the differentials~$\phi_0$ and $\phi_1$ will be completely different in general. In this way, the wall-crossing formula encodes how the finite length trajectories of a quadratic differential change as we vary the differential within the moduli space $\mathscr{Q}^\pm(\mathbb{S},\mathbb{M})$.

\subsection{Fock-Goncharov coordinates}

If $\phi$ is a general point in $\mathscr{Q}^\pm(\mathbb{S},\mathbb{M})$ and $\Delta\subset\mathbb{C}^*$ is a convex sector whose boundary rays are non-active, then Theorem~\ref{thm:introWCF} says that the BPS~automorphism $\mathbf{S}_\phi(\Delta)$ is invariant under small deformations of the differential. Our main result will allow us to compute this invariant explicitly using Fock-Goncharov coordinates, thereby making rigorous sense of a proposal of Gaiotto Moore, and Neitzke~\cite{GMN}. As a consequence, we will also see that $\mathbf{S}_\phi(\Delta)$ is a birational automorphism of~$\mathbb{T}_-$. This is closely related to the work~\cite{GoncharovShen}.

To explain this result in more detail, we consider triangulations of a marked bordered surface. Precisely, we define an ideal triangulation of~$(\mathbb{S},\mathbb{M})$ to be a triangulation of~$\mathbb{S}$ whose vertices are exactly the points of~$\mathbb{M}$. For technical reasons, we will also consider a variant of this concept called a tagged triangulation~\cite{FST}. Assuming $(\mathbb{S},\mathbb{M})$ is not a closed surface with exactly one marked point, any two tagged triangulations of~$(\mathbb{S},\mathbb{M})$ are related by a sequence of elementary operations called flips. Roughly speaking, a flip is an operation that removes an arc of the triangulation and replaces it by the unique different arc that results in a new triangulation.

In their seminal paper on higher Teichm\"uller theory~\cite{FockGoncharov1}, Fock and Goncharov introduced a moduli space $\mathscr{X}(\mathbb{S},\mathbb{M})$ parametrizing flat $\PGL_2(\mathbb{C})$-connections on the punctured surface $\mathbb{S}\setminus\mathbb{M}$ with additional framing data. If $\tau$ is any tagged triangulation of~$(\mathbb{S},\mathbb{M})$, let us write $\Gamma_\tau\cong\mathbb{Z}^n$ for the lattice spanned by the arcs of~$\tau$. Then there is a birational map 
\[
X_\tau:\mathscr{X}(\mathbb{S},\mathbb{M})\dashrightarrow\mathbb{T}_\tau\coloneqq\Hom_{\mathbb{Z}}(\Gamma_\tau,\mathbb{C}^*)
\]
from the moduli space to a complex algebraic torus~$\mathbb{T}_\tau\cong(\mathbb{C}^*)^n$. The components of this map corresponding to the arcs of~$\tau$ are known as Fock-Goncharov coordinates.

Let $(\mathbb{S},\mathbb{M})$ be a marked bordered surface which is not a closed surface with exactly one marked point. If $\phi$ is a general point in the space $\mathscr{Q}^\pm(\mathbb{S},\mathbb{M})$ and $\Delta\subset\mathbb{C}^*$ is a convex sector whose boundary rays are non-active with phases~$\theta_1$ and~$\theta_2$, then we will see that each rotated differential $e^{-2i\theta_j}\cdot\phi$ determines a tagged triangulation~$\tau_j$ of~$(\mathbb{S},\mathbb{M})$. In this case the tori $\mathbb{T}_{\tau_j}$ are naturally identified with $\mathbb{T}_+=\Hom_{\mathbb{Z}}(\Gamma_\phi,\mathbb{C}^*)$, and hence we can think of the Fock-Goncharov coordinates as providing birational maps 
\begin{equation}
\label{eqn:introsametarget}
X_{\tau_j}:\mathscr{X}(\mathbb{S},\mathbb{M})\dashrightarrow\mathbb{T}_+.
\end{equation}
As we explain below, the twisted torus $\mathbb{T}_-$ is a torsor for~$\mathbb{T}_+$, and hence these objects can be identified after choosing a basepoint in~$\mathbb{T}_-$.

\begin{theorem}
\label{thm:intromain}
Take notation as in the last paragraph. Then 
\begin{enumerate}
\item There is a distinguished basepoint $\xi\in\mathbb{T}_-$ such that $\xi(\gamma)=-1$ if $\gamma\in\Gamma_\phi$ is the class of a non-closed saddle connection and $\xi(\gamma)=+1$ if $\gamma$ is the class of a closed saddle connection.
\item $\mathbf{S}_\phi(\Delta)$ extends to a birational automorphism of~$\mathbb{T}_-$. If we use the basepoint $\xi$ to identify~$\mathbb{T}_-$ with~$\mathbb{T}_+$, then this is the birational automorphism of~$\mathbb{T}_+$ relating the maps~\eqref{eqn:introsametarget}.
\end{enumerate}
\end{theorem}

We note that Theorem~\ref{thm:main} provides a means of computing the BPS automorphism~$\mathbf{S}_\phi(\Delta)$. Indeed, the tagged triangulations~$\tau_j$ are related by a sequence of flips. For each pair of tagged triangulations related by a flip, the associated Fock-Goncharov coordinates are related by a well known coordinate transformation. By composing these coordinate transformations, one obtains the transformation~$\mathbf{S}_\phi(\Delta)$ explicitly.

\subsection{Application to WKB analysis}

Let us conclude this introduction by describing an application of the above results to exact WKB analysis. This material will not appear elsewhere in this paper and can be skipped. The ideas of this subsection are used in the sequel paper~\cite{Allegretti19b} to solve a class of Riemann-Hilbert problems formulated in~\cite{Bridgeland19}. These ideas originated in a slightly different setting in the work of Gaiotto, Moore, and Neitzke~\cite{GMN}.

The theory of exact WKB analysis is a tool that can be used to construct exact solutions of Schr\"odinger's equation for small values of the Planck constant~$\hbar$. In exact WKB~analysis, the Voros~symbols are certain formal series in~$\hbar$ which appear when explicitly calculating the monodromy group of Schr\"odinger's equation. As explained in~\cite{Allegretti19a}, one can associate a Voros symbol to a general point $\phi\in\mathscr{Q}^\pm(\mathbb{S},\mathbb{M})$ and class~$\gamma\in\Gamma_\phi$. This formal series has the form 
\[
Y_{\phi,\gamma}(\hbar)=e^{-Z_\phi(\gamma)/\hbar}\cdot\sum_{k=0}^\infty\hbar^kf_k
\]
for some coefficients $f_k$.

In general, the Voros symbol diverges and therefore does not define a holomorphic function of~$\hbar$. To get around this difficulty, one must use Borel resummation. In the Borel resummation method, one starts with a (possibly divergent) formal power series $f(\hbar)=\sum_{k=0}^\infty\hbar^kf_k$ which is Borel~summable in the sense defined in~\cite{Allegretti19a,IwakiNakanishi}. One then constructs a function $\mathcal{S}[f]$ called the Borel~sum of~$f$, which is holomorphic in a small sectorial neighborhood of $\hbar=0$ and has $f$ as its asymptotic expansion as $\hbar\rightarrow0$. The Borel sum is given by the expression 
\begin{equation}
\label{eqn:Borelsum}
\mathcal{S}[f](\hbar)=f_0+\int_0^\infty e^{-y/\hbar}f_B(y)dy
\end{equation}
where $f_B(y)$ denotes the Borel transform of~$f$ (see~\cite{IwakiNakanishi} for details). The series appearing in the definition of the Voros symbol is known to be Borel summable~\cite{IwakiNakanishi} provided there is no active ray of phase zero for the differential~$\phi$. In this case, we can think of the Voros symbol as the function given by $\mathcal{Y}_{\phi,\gamma}(\hbar)=e^{-Z_\phi(\gamma)/\hbar}\cdot\mathcal{S}[f](\hbar)$.

If there is an active ray of phase zero, then the Voros symbol may not be Borel summable. This happens because the Borel transform develops singularities along the positive real axis, making the integral in~\eqref{eqn:Borelsum} undefined. In this case, we can choose a small angle $\theta$, and define a modified Borel sum $\mathcal{S}_\theta[f]$ by the same formula~\eqref{eqn:Borelsum} where now the integral is taken along the ray $y=re^{i\theta}$,~$r\geq0$ in~$\mathbb{C}$. The resulting function $\mathcal{S}_\theta[f]$ is again holomorphic in a sectorial neighborhood of~$\hbar=0$ and has~$f$ as its asymptotic expansion as $\hbar\rightarrow0$. However, in this case an interesting Stokes phenomenon can occur: for sufficiently small angles $\theta>0$, the two functions $\mathcal{S}_{\pm\theta}[f]$ have the same asymptotic expansion as $\hbar\rightarrow0$ but are not the same function.

The Stokes phenomenon for Voros symbols was studied by Delabaere, Dillinger, and Pham~\cite{DDP} who considered the situation where there is a unique finite-length trajectory whose class $\gamma$ satisfies $Z_\phi(\gamma)\in\mathbb{R}$, which moreover is a non-closed saddle connection. For sufficiently small $\theta>0$, the functions $\mathcal{Y}_{\phi,\beta}^{(\pm\theta)}(\hbar)=e^{-Z_\phi(\gamma)/\hbar}\cdot\mathcal{S}_{\pm\theta}[f](\hbar)$ are defined, and Delabaere, Dillinger, and Pham showed that 
\[
\mathcal{Y}_{\phi,\beta}^{(-\theta)}=\mathcal{Y}_{\phi,\beta}^{(+\theta)}\cdot\left(1+\mathcal{Y}_{\phi,\gamma}^{(+\theta)}\right)^{\langle\beta,\gamma\rangle}
\]
as analytic functions where $\gamma\in\Gamma_\phi$ is the class of the saddle connection. The transformation relating the functions $\mathcal{Y}_{\phi,\beta}^{(+\theta)}(\hbar)$ to $\mathcal{Y}_{\phi,\beta}^{(-\theta)}(\hbar)$ is known as a Stokes automorphism. Related results concerning Stokes automorphisms for Voros symbols have been obtained in~\cite{AIT,IwakiNakanishi}.

Using Theorem~\ref{thm:intromain}, we can give a completely general description of the Stokes automorphisms for Voros symbols. Indeed, suppose $\phi\in\mathscr{Q}^\pm(\mathbb{S},\mathbb{M})$ has an active ray of phase zero. For any $\varepsilon>0$, we can choose $0<\theta<\varepsilon$ so that the rays $r_\pm=\mathbb{R}_{>0}\cdot e^{\pm2i\theta}$ are non-active. Let $\tau_\pm$ be the corresponding tagged triangulations of~$(\mathbb{S},\mathbb{M})$, and for any $\gamma\in\Gamma_\phi$, let $x_\gamma:\mathbb{T}_+\rightarrow\mathbb{C}^*$ be the character $x_\gamma(g)=g(\gamma)$. Then the main result of~\cite{Allegretti19a} says that there exists a one parameter family of points $\mathcal{L}_\phi(\hbar)\in\mathscr{X}(\mathbb{S},\mathbb{M})$ such that 
\[
\mathcal{Y}_{\phi,\gamma}^{\pm\theta}(\hbar)=x_\gamma\circ X_{\tau_\pm}(\mathcal{L}_\phi(\hbar)).
\]
It follows from Theorem~\ref{thm:intromain} that after identifying $\mathbb{T}_+$ with~$\mathbb{T}_-$ using the canonical basepoint~$\xi$, the Stokes automorphism is precisely the BPS automorphism $\mathbf{S}_\phi(\Delta)$ where $\Delta\subset\mathbb{C}^*$ is the sector having boundary rays~$r_\pm$. In particular, this Stokes automorphism is a product of (perhaps infinitely many) factors $\mathbf{S}_\phi(\ell)$ associated to active rays $\ell\subset\Delta$.

\subsection*{Acknowledgements.}
I thank Tom~Bridgeland for encouraging me to write this paper and for suggesting a strategy for the proof of the main result. While working on this project, I benefited from conversations and correspondence with David~Aulicino, Ben~Davison, Kohei~Iwaki, Sven~Meinhardt, Andrew~Neitzke, and Anton~Zorich. This project was conceived while I was in residence at the Mathematical Sciences Research Institute in Berkeley, California, during the Fall~2019 semester, partially supported by National Science Foundation grant DMS-1440140.

\section{Quadratic differentials on Riemann surfaces}

In this section, we review some basic definitions from the theory of quadratic differentials and describe our geometric setup. None of this material is new, and most of it can be found in~\cite{BridgelandSmith}.

\subsection{GMN differentials}

Throughout this paper, we will write $S$ for a compact Riemann surface of genus~$g\geq0$ and write $\omega_S$ for its holomorphic cotangent bundle. Then a meromorphic \emph{quadratic differential} on~$S$ is defined to be a meromorphic section of $\omega_S^{\otimes2}$. In terms of a local coordinate~$z$ on~$S$, such a section $\phi$ can be written 
\[
\phi(z)=\varphi(z)dz^2
\]
where $\varphi(z)$ is a meromorphic function in the local coordinate.

By a \emph{critical point} of a quadratic differential~$\phi$, we mean either a zero or pole of~$\phi$. We will write $\Crit(\phi)$ for the set of all critical points of~$\phi$. Such a point is called a \emph{finite critical point} if it is a zero or a simple pole and an \emph{infinite critical point} otherwise. Thus the set of all critical points is a disjoint union 
\[
\Crit(\phi)=\Crit_{<\infty}(\phi)\cup\Crit_\infty(\phi)
\]
where $\Crit_{<\infty}(\phi)$ is the set of finite critical points and $\Crit_\infty(\phi)$ is the set of infinite critical points.

In this paper, we will be concerned with quadratic differentials of the following special type.

\begin{definition}
\label{def:GMNdifferential}
A \emph{Gaiotto-Moore-Neitzke (GMN) differential} is a meromorphic quadratic differential $\phi$ on a compact, connected Riemann surface~$S$ satisfying the following conditions:
\begin{enumerate}
\item $\phi$ has no zero of order $>1$.
\item $\phi$ has at least one pole.
\item $\phi$ has at least one finite critical point.
\end{enumerate}
A GMN differential is said to be \emph{complete} if in addition it has no simple poles so that every pole has order~$\geq2$.
\end{definition}

\subsection{The canonical double cover}

Let $\phi$ be a GMN differential on the compact Riemann surface~$S$. Let $p_i\in S$ be the poles of $\phi$, and let $m_i$ be the order of the pole~$p_i$. Then we can think of the differential $\phi$ as a holomorphic section $s_\phi$ of $\omega_S(E)^{\otimes2}$ where $E$ is the divisor 
\[
E=\sum_i\left\lceil\frac{m_i}{2}\right\rceil p_i
\]
and $s_\phi$ has simple zeros at both the zeros and odd order poles of~$\phi$. We can then define the \emph{canonical double cover} as the subspace $\Sigma_\phi$ of the total space of $\omega_S(E)$ cut out by the equation 
\[
\lambda^2=s_\phi(p)
\]
in the fiber over each point $p\in S$. The space $\Sigma_\phi$ defined in this way is a Riemann surface, and the obvious projection $\pi:\Sigma_\phi\rightarrow S$ is a covering map, branched precisely over the set of zeros and odd order poles of~$\phi$. Since $\phi$ has at least one finite critical point, this covering map has at least one branch point and hence $\Sigma_\phi$ is connected.

By construction, the canonical double cover comes with a tautological section $\lambda$ of $\pi^*(\omega_S(E))$ satisfying $\pi^*(s_\phi)=\lambda\otimes\lambda$. It can be viewed alternatively as a meromorphic 1-form on~$\Sigma_\phi$. Note that if $p$ is a simple zero of~$\phi$ then the cover $\pi:\Sigma_\phi\rightarrow S$ is ramified at~$p$, and therefore the preimage $\pi^{-1}(p)$ consists of a single point. Later, we will need to know the following simple fact concerning behavior of $\lambda$ at such a point.

\begin{lemma}
\label{lem:doublezeros}
Suppose $p\in S$ is a simple zero of~$\phi$. Then the tautological 1-form $\lambda$ has a zero of order two at~$\pi^{-1}(p)$.
\end{lemma}

\begin{proof}
By Theorem~6.1 of~\cite{Strebel}, there exists a local coordinate $z$ defined in a neighborhood of~$p$ and satisfying $z(p)=0$ so that in this local coordinate we have 
\[
\phi(z)=c\cdot zdz^2
\]
for some $c\in\mathbb{C}^*$. Since $\pi:\Sigma_\phi\rightarrow S$ is ramified at~$p$, we can choose a local coordinate $w$ in a neighborhood of $\pi^{-1}(p)$ so that the projection map is given by $\pi^*(z)=w^2$. Then 
\[
(\pi^*\phi)(w)=c\cdot w^2d(w^2)^2=4c\cdot w^4dw^2,
\]
and therefore $\lambda$ has a zero of order $\frac{1}{2}(4)=2$.
\end{proof}

\subsection{The period map}

For convenience, we use the notation 
\[
\Sigma_\phi^\circ=\Sigma_\phi\setminus\pi^{-1}\Crit_\infty(\phi).
\]
There is a natural involution $\tau:\Sigma_\phi\rightarrow\Sigma_\phi$ exchanging the two sheets of the double cover and commuting with the covering map~$\pi$. We will write 
\[
\Gamma_\phi=\left\{\sigma\in H_1(\Sigma_\phi^\circ,\mathbb{Z}):\tau(\sigma)=-\sigma\right\}
\]
for the anti-invariant part of the first homology $H_1(\Sigma_\phi^\circ,\mathbb{Z})$ with respect to this covering involution. It is a free abelian group of finite rank. By integrating the canonical 1-form $\lambda$ around cycles in~$\Gamma_\phi$, we obtain a group homomorphism 
\[
Z_\phi:\Gamma_\phi\rightarrow\mathbb{C}, \quad Z_\phi(\gamma)=\int_\gamma\lambda,
\]
called the \emph{period map} for the differential~$\phi$.

\subsection{Trajectories}

Let $\phi$ be a meromorphic quadratic differential on a compact Riemann surface~$S$. Near any point of $S\setminus\Crit(\phi)$, there is a distinguished local coordinate $w$, unique up to transformations of the form $w\mapsto\pm w+\text{constant}$, with respect to which the quadratic differential $\phi$ is given by 
\[
\phi(w)=dw\otimes dw.
\]
Indeed, if we have $\phi(z)=\varphi(z)dz^2$ for some local coordinate $z$, then $w$ is given by $w=\int\sqrt{\varphi(z)}dz$ for some choice of the square root. These distinguished local coordinates determine two structures on~$S\setminus\Crit(\phi)$. The first is a flat metric defined by pulling back the Euclidean metric on~$\mathbb{C}$ by the distinguished local coordinates. The other structure is the \emph{horizontal foliation}, defined as the foliation by the curves $\Im(w)=\text{constant}$.

By a \emph{straight arc} in~$S$, we mean a smooth path $\alpha:I\rightarrow S\setminus\Crit(\phi)$, defined on an open interval $I\subset\mathbb{R}$, which makes a constant angle $\pi\theta$ with the leaves of the horizontal foliation. By convention, straight arcs will be parametrized by arc length in the flat metric induced by the distinguished local coordinates, and two straight arcs will be regarded as the same if they are related by a reparametrization of the form $t\mapsto\pm t+\text{constant}$. The phase $\theta$ of a straight arc is well defined in~$\mathbb{R}/\mathbb{Z}$, and a straight arc of phase $\theta=0$ is said to be \emph{horizontal}.

A straight arc is called a \emph{trajectory} if it is not the restriction of a straight arc defined on a larger interval. Thus a horizontal trajectory is the same thing as a leaf of the horizontal foliation. A \emph{saddle connection} is a trajectory of any phase whose domain of definition is a finite length interval. A saddle connection is said to be \emph{closed} if its endpoints coincide. Note that if a trajectory intersects itself in $S\setminus\Crit(\phi)$ then it must be periodic and have domain $I=\mathbb{R}$. In this case it is called a \emph{closed trajectory}. By a \emph{finite-length trajectory}, we mean either a saddle connection or a closed trajectory.

Any closed trajectory is contained in an annular region foliated by homotopic closed trajectories. Such a region will be called a \emph{cylinder}. The boundary of a cylinder has two components. Typically, each component is composed of saddle connections, but it can also happen that the cylinder consists of closed trajectories encircling a single pole of order two, which forms one of the boundary components. In the latter case the cylinder is said to be \emph{degenerate}.

\subsection{Homology classes}

Consider a GMN differential~$\phi$ on a compact Riemann surface~$S$. If $\alpha:I\rightarrow S$ is a finite-length trajectory for~$\phi$ which is horizontal, then we can consider the preimage $\widehat{\alpha}=\pi^{-1}(\alpha)$ of this trajectory in the canonical double cover~$\Sigma_\phi$. This preimage is a closed curve which may be disconnected if $\alpha$ is a closed trajectory. As we have seen, there is a canonical 1-form~$\lambda$ on the double cover with the property that $\pi^*(\phi)=\lambda\otimes\lambda$. We can endow the closed curve~$\widehat{\alpha}$ with a canonical orientation by requiring that $\lambda$ evaluated on a tangent vector to the oriented curve be real and positive.

Similarly, if $\alpha:I\rightarrow S$ is a finite-length trajectory with some nonzero phase $\theta$, then we can lift~$\alpha$ to a closed curve $\widehat{\alpha}$ in the canonical double cover. We can once again endow this closed curve with an orientation, but in this case, we require that $\lambda$ evaluated on a tangent vector to the oriented curve have positive imaginary part.

Thus we associate, to any finite-length trajectory $\alpha$ of the differential, a corresponding cycle $\widehat{\alpha}$ in the canonical double cover. The covering involution reverses the orientation of this cycle, and so we obtain an anti-invariant class $\widehat{\alpha}\in\Gamma_\phi$ in homology, which we call the class of $\alpha$.

\subsection{Finite-length trajectories}

The wall-crossing formula considered in this paper describes the jumping behavior of a certain integer counting finite length trajectories with a given homology class in~$\Gamma_\phi$. To define this integer, we will impose a genericity assumption on the differential~$\phi$. Namely, we say that a GMN differential $\phi$ is \emph{generic} if, for any two classes $\gamma_1$,~$\gamma_2\in\Gamma_\phi$, we have 
\[
\mathbb{R}\cdot Z_\phi(\gamma_1)=\mathbb{R}\cdot Z_\phi(\gamma_2) \implies \mathbb{Z}\cdot\gamma_1=\mathbb{Z}\cdot\gamma_2.
\]
Let $\phi$ be a generic GMN differential. Note that a closed trajectory lies in a cylinder of trajectories of the same phase and any other closed trajectory in this cylinder has the same class in~$\Gamma_\phi$. Therefore we can speak of the class of the cylinder.

\begin{definition}
The \emph{BPS invariant} associated to a generic GMN differential~$\phi$ and class $\gamma\in\Gamma_\phi$ is the integer 
\begin{align*}
\Omega_\phi(\gamma) &= |\{\text{non-closed saddle connections of class $\pm\gamma$}\}| \\
&\quad - 2\cdot|\{\text{nondegenerate cylinders of class $\pm\gamma$}\}|.
\end{align*}
\end{definition}

Note in particular that a closed saddle trajectory makes no contribution to $\Omega_\phi(\gamma)$ though it may appear as a boundary component of a nondegenerate cylinder. As we explain in Appendix~\ref{sec:TheMotivicWallCrossingFormula}, this integer $\Omega_\phi(\gamma)$ coincides with the BPS invariant defined in a general categorical setting in~\cite{KontsevichSoibelman1}. It is for this reason that these integers $\Omega_\phi(\gamma)$ obey the wall-crossing formula.

\section{From quadratic differentials to ideal triangulations}
\label{sec:FromQuadraticDifferentialsToIdealTriangulations}

In this section, we review some additional background material on quadratic differentials. We define a moduli space parametrizing GMN differentials and explain how a general point in this space determines an ideal triangulation of an associated marked bordered surface.

\subsection{Critical points}

Let $\phi$ be a GMN differential. Suppose that $p$ is a finite critical point of~$\phi$, and let $k$ be the order of the singularity so that $k=1$ if $p$ is a simple zero and~$k=-1$ if $p$ is a simple pole. By the results of~\cite{Strebel}, there is a local coordinate $t$ defined in a neighborhood of~$p$ such that 
\[
\phi(t)=\left(\frac{k+2}{2}\right)^2\cdot t^kdt^2.
\]
Away from the point~$p$, the function $w=t^{\frac{k+2}{2}}$ is a distinguished local coordinate. The horizontal trajectories in a neighborhood of~$p$ are illustrated in Figure~\ref{fig:finitecriticalpoint}.

\begin{figure}[ht]
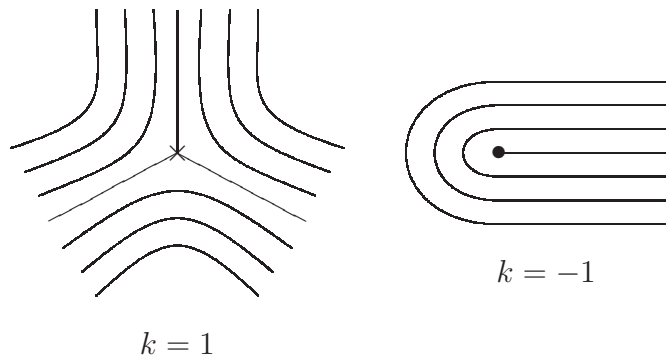

\begin{center}
\[
\xy /l3pc/:
(1,0)*{}="O";  
(-0.35,0.72)*{}="U";  
(-0.75,-0.05)*{}="X1";
(-0.6,0.2)*{}="X2"; 
(-0.45,0.45)*{}="X3"; 
(-0.2,1)*{}="X4";  
(0,1.25)*{}="X5";  
(0.15,1.5)*{}="X6"; 
(1.85,1.5)*{}="Y1";
(2,1.25)*{}="Y2";
(2.2,1)*{}="Y3";
(2.45,0.45)*{}="Y4"; 
(2.6,0.2)*{}="Y5";
(2.75,-0.05)*{}="Y6"; 
(1.85,-1.5)*{}="Z1";
(1.55,-1.5)*{}="Z2"; 
(1.25,-1.5)*{}="Z3";
(0.75,-1.5)*{}="Z4";
(0.45,-1.5)*{}="Z5";
(0.15,-1.5)*{}="Z6";
(2.35,0.72)*{}="V";  
(1,-1.5)*{}="W"; 
"O";"U" **\dir{-};  
"O";"V" **\dir{-}; 
"O";"W" **\dir{-}; 
"X4";"Y3" **\crv{(0.9,0.2) & (1.1,0.2)};
"X5";"Y2" **\crv{(0.9,0.5) & (1.1,0.5)}; 
"X6";"Y1" **\crv{(0.9,0.8) & (1.1,0.8)};
"Y4";"Z3" **\crv{(1.35,0) & (1.15,0)};
"Y5";"Z2" **\crv{(1.5,-0.2) & (1.5,-0.3)};
"Y6";"Z1" **\crv{(1.65,-0.4) & (1.85,-0.6)};
"Z4";"X3" **\crv{(0.85,0) & (0.65,0)};
"Z5";"X2" **\crv{(0.5,-0.3) & (0.5,-0.2)};
"Z6";"X1" **\crv{(0.15,-0.6) & (0.35,-0.4)};
(1,0)*{\times};
(1,2)*{k=1};
\endxy
\quad
\xy /l3pc/:
(1,0)*{}="O"; 
(-0.75,0)*{}="U"; 
(-0.75,-0.75)*{}="U1";
(-0.75,-0.5)*{}="U2";
(-0.75,-0.25)*{}="U3"; 
(-0.75,0.25)*{}="U4";
(-0.75,0.5)*{}="U5";
(-0.75,0.75)*{}="U6";
(1,0.75)*{}="T6";
(1,0.5)*{}="T5";
(1,0.25)*{}="T4";
(1,-0.25)*{}="T3";
(1,-0.5)*{}="T2";
(1,-0.75)*{}="T1";
"O";"U" **\dir{-}; 
"T6";"U6" **\dir{-}; 
"T5";"U5" **\dir{-}; 
"T4";"U4" **\dir{-}; 
"T3";"U3" **\dir{-}; 
"T2";"U2" **\dir{-}; 
"T1";"U1" **\dir{-}; 
"T3";"T4" **\crv{(1.5,-0.25) & (1.5,0.25)};
"T2";"T5" **\crv{(1.9,-0.5) & (1.9,0.5)};
"T1";"T6" **\crv{(2.3,-0.75) & (2.3,0.75)};
(1,0)*{\bullet};
(0.5,1.25)*{k=-1};
\endxy
\]
\end{center}
\caption{The horizontal foliation near a finite critical point.\label{fig:finitecriticalpoint}}
\end{figure}

Similarly, if $p$ is a pole of~$\phi$ of order two, then there is a local coordinate $t$ defined in a neighborhood of~$p$ such that 
\[
\phi(t)=\frac{r}{t^2}dt^2
\]
for some well defined constant $r\in\mathbb{C}^*$. We define the \emph{residue} of $\phi$ at $p$ to be the quantity 
\[
\Res_p(\phi)=\pm4\pi i\sqrt{r},
\]
which is well defined up to a sign. Away from~$p$, any branch of the function $w=\sqrt{r}\log(t)$ is a distinguished local coordinate. The horizontal foliation can exhibit three possible behaviors in the $t$-plane depending on the value of the residue at~$p$: 
\begin{enumerate}
\item If $\Res_p(\phi)\in\mathbb{R}$, then the horizontal trajectories are concentric circles centered at the pole.
\item If $\Res_p(\phi)\in i\mathbb{R}$, then the horizontal trajectories are radial arcs emanating from the pole.
\item If $\Res_p(\phi)\not\in\mathbb{R}\cup i\mathbb{R}$, then the horizontal trajectories are logarithmic spirals that wrap around the pole.
\end{enumerate}
Figure~\ref{fig:doublepole} illustrates the three types of foliations.

\begin{figure}[ht]
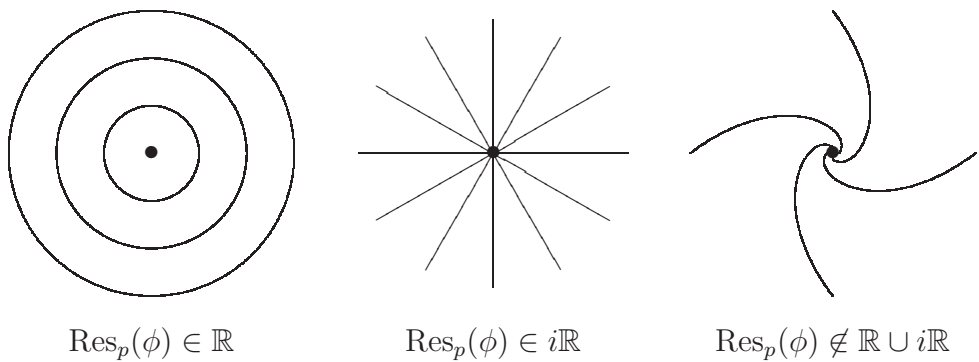

\begin{center}
\[
\xy /l1.5pc/:
(1,-3)*\xycircle(3,3){-};
(1,-2)*\xycircle(2,2){-};
(1,-1)*\xycircle(1,1){-};
(1,0)*{\bullet};
(1,4)*{\Res_p(\phi)\in\mathbb{R}};
\endxy
\qquad
\qquad
\xy /l1.5pc/:
{\xypolygon12"A"{~:{(2,2):}~>{}}};
{(1,0)\PATH~={**@{-}}'"A1"};
{(1,0)\PATH~={**@{-}}'"A2"};
{(1,0)\PATH~={**@{-}}'"A3"};
{(1,0)\PATH~={**@{-}}'"A4"};
{(1,0)\PATH~={**@{-}}'"A5"};
{(1,0)\PATH~={**@{-}}'"A6"};
{(1,0)\PATH~={**@{-}}'"A7"};
{(1,0)\PATH~={**@{-}}'"A8"};
{(1,0)\PATH~={**@{-}}'"A9"};
{(1,0)\PATH~={**@{-}}'"A10"};
{(1,0)\PATH~={**@{-}}'"A11"};
{(1,0)\PATH~={**@{-}}'"A12"};
(1,0)*{\bullet};
(1,4)*{\Res_p(\phi)\in i\mathbb{R}};
\endxy
\qquad
\xy /l1.5pc/:
(1,0)*{\bullet};
(1.13,0);(-2,0) **\crv{(1.5,0.5) & (0,1.5)};
(1,-0.13);(1,3) **\crv{(1.5,-0.5) & (2.5,1)};
(0.87,0);(4,0) **\crv{(0.5,-0.5) & (2,-1.5)};
(1,0.13);(1,-3) **\crv{(0.5,0.5) & (-0.5,-1)};
(1,4)*{\Res_p(\phi)\not\in\mathbb{R}\cup i\mathbb{R}};
\endxy
\]
\end{center}
\caption{The horizontal foliation near a pole of order two.\label{fig:doublepole}}
\end{figure}

Finally, suppose that $p$ is a pole of order $m\geq3$. By the results of~\cite{Strebel}, there is a local coordinate~$t$ such that 
\[
\phi(t)=\left(\frac{2-m}{2}t^{-m/2}+O(t^{-1})\right)^2dt^2 \quad\text{as $t\rightarrow0$}.
\]
In this case, one can show that there are $m-2$ distinguished tangent directions at~$p$. There is a neighborhood $U$ of $p$ such that any horizontal trajectory that enters $U$ eventually tends to $p$ along one of these tangent directions. We illustrate this for small values of~$m$ in Figure~\ref{fig:higherorderpole}.

\begin{figure}[ht]
\begin{center}
\[
\xy /l3pc/:
{\xypolygon3"T"{~:{(2,0):}~>{}}},
{\xypolygon3"S"{~:{(1.5,0):}~>{}}},
{\xypolygon3"R"{~:{(1,0):}~>{}}},
(1,0)*{}="O"; 
(-0.35,0.72)*{}="U"; 
(2.35,0.72)*{}="V"; 
(1,-1.5)*{}="W"; 
"O";"U" **\dir{-}; 
"O";"V" **\dir{-}; 
"O";"W" **\dir{-}; 
"O";"T1" **\crv{(2,0.75) & (2.25,1.85)};
"O";"T1" **\crv{(0,0.75) & (-0.25,1.85)};
"O";"T2" **\crv{(0,0.4) & (-1.2,0.25)};
"O";"T2" **\crv{(1,-1) & (0,-2.2)};
"O";"T3" **\crv{(1,-1) & (2,-2.2)};
"O";"T3" **\crv{(2,0.4) & (3.2,0.25)};
"O";"S1" **\crv{(1.75,0.56) & (2,1.5)};
"O";"S1" **\crv{(0.25,0.56) & (0,1.5)};
"O";"S2" **\crv{(0,0.3) & (-0.9,0.19)};
"O";"S2" **\crv{(0.9,-0.8) & (0.3,-1.7)};
"O";"S3" **\crv{(2,0.3) & (2.9,0.19)};
"O";"S3" **\crv{(1.1,-0.8) & (1.7,-1.7)};
"O";"R1" **\crv{(1.5,0.5) & (1.75,1)};
"O";"R1" **\crv{(0.5,0.5) & (0.25,1)};
"O";"R2" **\crv{(0.5,0.1) & (-0.3,0.25)};
"O";"R2" **\crv{(0.75,-0.8) & (0.5,-1)};
"O";"R3" **\crv{(1.5,0.1) & (2.3,0.25)};
"O";"R3" **\crv{(1.25,-0.8) & (1.5,-1)};
(1,0)*{\bullet};
(1,2.25)*{m=5};
\endxy
\qquad
\xy /l3pc/:
{\xypolygon4"A"{~:{(2,0):}~>{}}},
{\xypolygon4"B"{~:{(1.5,0):}~>{}}},
{\xypolygon4"C"{~:{(1,0):}~>{}}},
(1,0)*{}="O"; 
(1,-1.75)*{}="T"; 
(-0.75,0)*{}="U"; 
(1,1.75)*{}="V"; 
(2.75,0)*{}="W"; 
"O";"T" **\dir{-}; 
"O";"U" **\dir{-}; 
"O";"V" **\dir{-}; 
"O";"W" **\dir{-}; 
"O";"A1" **\crv{(1,1.5) & (1.5,2.25)};
"O";"A1" **\crv{(2.5,0) & (3.25,0.5)};
"O";"A2" **\crv{(1,1.5) & (0.5,2.25)};
"O";"A2" **\crv{(-0.5,0) & (-1.25,0.5)};
"O";"A3" **\crv{(1,-1.5) & (0.5,-2.25)};
"O";"A3" **\crv{(-0.5,0) & (-1.25,-0.5)};
"O";"A4" **\crv{(1,-1.5) & (1.5,-2.25)};
"O";"A4" **\crv{(2.5,0) & (3.25,-0.5)};
"O";"B1" **\crv{(1,1) & (1.5,1.6)};
"O";"B1" **\crv{(2,0) & (2.6,0.5)};
"O";"B2" **\crv{(1,1) & (0.5,1.6)};
"O";"B2" **\crv{(0,0) & (-0.6,0.5)};
"O";"B3" **\crv{(1,-1) & (0.5,-1.6)};
"O";"B3" **\crv{(0,0) & (-0.6,-0.5)};
"O";"B4" **\crv{(1,-1) & (1.5,-1.6)};
"O";"B4" **\crv{(2,0) & (2.6,-0.5)};
"O";"C1" **\crv{(1,0.75) & (1.25,1)};
"O";"C1" **\crv{(1.75,0) & (2,0.25)};
"O";"C2" **\crv{(1,0.75) & (0.75,1)};
"O";"C2" **\crv{(0.25,0) & (0,0.25)};
"O";"C3" **\crv{(1,-0.75) & (0.75,-1)};
"O";"C3" **\crv{(0.25,0) & (0,-0.25)};
"O";"C4" **\crv{(1,-0.75) & (1.25,-1)};
"O";"C4" **\crv{(1.75,0) & (2,-0.25)};
(1,0)*{\bullet};
(1,2.25)*{m=6};
\endxy
\qquad
\dots
\]
\end{center}
\caption{The horizontal foliation near a pole of order $m\geq3$.\label{fig:higherorderpole}}
\end{figure}

\subsection{Moduli spaces}

By a \emph{marked bordered surface}, we mean a compact, connected, oriented surface $\mathbb{S}$ with boundary together with a nonempty finite set $\mathbb{M}\subset\mathbb{S}$ of marked points such that every boundary component of~$\mathbb{S}$ contains at least one marked point. A marked point in the interior of~$\mathbb{S}$ is called a \emph{puncture}, and the set of all punctures is denoted $\mathbb{P}\subset\mathbb{M}$. An isomorphism of marked bordered surfaces $(\mathbb{S}_1,\mathbb{M}_1)$ and $(\mathbb{S}_2,\mathbb{M}_2)$ is an orientation preserving diffeomorphism $f:\mathbb{S}_1\rightarrow\mathbb{S}_2$ which induces a bijection $\mathbb{M}_1\rightarrow\mathbb{M}_2$. Two such isomorphisms are said to be isotopic if they are related by an isotopy through isomorphisms.

A pair $(S,\phi)$ consisting of a compact Riemann surface~$S$ and a GMN differential~$\phi$ on~$S$ determines an associated marked bordered surface $(\mathbb{S},\mathbb{M})$ by the following construction. To define the surface~$\mathbb{S}$, we perform an oriented real blow up of the Riemann surface~$S$ at each pole of~$\phi$ of order $\geq3$. As we have seen, there are finitely many distinguished tangent directions at each pole of order~$\geq3$. These determine points on the boundary of~$\mathbb{S}$, and we define~$\mathbb{M}$ to be the set consisting of these points together with the poles of order~$\leq2$ regarded as punctures.

Let us now fix a marked bordered surface $(\mathbb{S},\mathbb{M})$. If $(S,\phi)$ is a pair consisting of a compact Riemann surface~$S$ and a GMN~differential~$\phi$ on~$S$, then we define a \emph{marking} of the pair $(S,\phi)$ by $(\mathbb{S},\mathbb{M})$ to be an isotopy class of isomorphisms from $(\mathbb{S},\mathbb{M})$ to the marked bordered surface determined by~$(S,\phi)$. A \emph{marked GMN differential} is a triple $(S,\phi,\theta)$ where $S$ is a compact Riemann surface equipped with a GMN~differential~$\phi$ and $\theta$ is a marking of the pair $(S,\phi)$ by~$(\mathbb{S},\mathbb{M})$. We will consider two such triples $(S_1,\phi_1,\theta_1)$ and $(S_2,\phi_2,\theta_2)$ to be equivalent if there is an isomorphism $f:S_1\rightarrow S_2$ of Riemann surfaces satisfying $f^*(\phi_2)=\phi_1$ and commuting with the markings~$\theta_i$ in the obvious way. We denote by $\mathscr{Q}(\mathbb{S},\mathbb{M})$ the moduli space of equivalence classes of marked GMN differentials. Assuming $\mathbb{S}$ is not a genus zero surface with $|\mathbb{M}|\leq2$, Proposition~6.2 of~\cite{Allegretti19b} implies that $\mathscr{Q}(\mathbb{S},\mathbb{M})$ has the structure of a complex manifold.

In fact, it will be important to enhance this construction slightly and consider a moduli space parametrizing marked GMN differentials with additional data. Recall that if $\phi$ is a GMN~differential having a pole of order two at~$p$, then the residue $\Res_p(\phi)$ is well defined up to sign. We define a signing for~$\phi$ to be a choice of sign for the residue at each pole of order two. There is a branched cover 
\[
\mathscr{Q}^\pm(\mathbb{S},\mathbb{M})\rightarrow\mathscr{Q}(\mathbb{S},\mathbb{M})
\]
of degree $2^{|\mathbb{P}|}$ obtained by choosing a signing for each differential in $\mathscr{Q}(\mathbb{S},\mathbb{M})$ having a pole of order two. It is branched precisely over the locus of differentials having simple poles. Assuming once again that $\mathbb{S}$ is not a genus zero surface with $|\mathbb{M}|\leq2$, Proposition~6.3 of~\cite{Allegretti19b} gives $\mathscr{Q}^\pm(\mathbb{S},\mathbb{M})$ the structure of a complex manifold.

\subsection{Horizontal strip decomposition}

In general, if $\phi$ is a GMN differential on a compact Riemann surface~$S$, then any horizontal trajectory of~$\phi$ is necessarily one of the following (see~\cite{Strebel}, Sections~9--11):
\begin{enumerate}
\item A \emph{saddle trajectory}, which connects two finite critical points of~$\phi$.
\item A \emph{separating trajectory}, which connects a finite and an infinite critical point of~$\phi$.
\item A \emph{generic trajectory}, which connects two infinite critical points of~$\phi$.
\item A \emph{closed trajectory}, which is a simple closed curve in $S\setminus\Crit(\phi)$.
\item A \emph{recurrent trajectory}, which has a limit set with nonempty interior in~$S$.
\end{enumerate}
We will be particularly interested in quadratic differentials having no saddle trajectories. Such a differential is said to be \emph{saddle-free}.

Suppose $\phi$ is a saddle-free GMN differential with $\Crit_\infty(\phi)\neq\emptyset$. Then according to Lemma~3.1 of~\cite{BridgelandSmith}, this differential $\phi$ has no closed or recurrent trajectories. Moreover, since $\phi$ has at most finitely many zeros, there can be at most finitely many separating trajectories. If we remove these separating trajectories from~$S$, then the remaining open surface splits as a union of connected components, and each component is one of the following:
\begin{enumerate}
\item A \emph{horizontal strip}, which is is a maximal domain in~$S$ that corresponds, via the distinguished local coordinate, to a region 
\[
\{w\in\mathbb{C}:a<\Im(w)<b\}\subset\mathbb{C}.
\]
Every trajectory in a horizontal strip is generic, connecting two (not necessarily distinct) poles of~$\phi$. Each component of the boundary is composed of separating trajectories.
\item A \emph{half plane}, which is a maximal domain in~$S$ that corresponds, via the distinguished local coordinate, to a region 
\[
\{w\in\mathbb{C}:\Im(w)>0\}\subset\mathbb{C}.
\]
The trajectories in a half plane are generic, connecting a fixed pole of order $>2$ to itself. The boundary is composed of separating trajectories.
\end{enumerate}
This decomposition of the surface into horizontal strips and half planes is called the \emph{horizontal strip decomposition}.

\subsection{Ideal triangulations}

If $(\mathbb{S},\mathbb{M})$ is a marked bordered surface, then an \emph{arc} on $(\mathbb{S},\mathbb{M})$ is defined as a smooth path $\gamma$ in~$\mathbb{S}$ connecting points of~$\mathbb{M}$ whose interior lies in $\mathbb{S}\setminus\mathbb{M}$ and which has no self-intersections in its interior. In addition, we require that $\gamma$ is not homotopic, relative to its endpoints, to a single point or to a path in~$\partial\mathbb{S}$ whose interior contains no marked points. Two arcs are considered to be equivalent if they are homotopic relative to their endpoints, and they are \emph{compatible} if there exist arcs in their respective equivalence classes which do not intersect in the interior of the surface~$\mathbb{S}$.

An \emph{ideal triangulation} of~$(\mathbb{S},\mathbb{M})$ is defined to be a maximal set of pairwise compatible arcs, considered up to equivalence. When talking about an ideal triangulation~$T$ of~$(\mathbb{S},\mathbb{M})$, we always fix representatives for its arcs so that no two arcs intersect in the interior of~$\mathbb{S}$. Then a \emph{triangle} of~$T$ is defined to be the closure in~$\mathbb{S}$ of a component of the complement of the arcs of~$T$. Any triangle is homeomorphic to a disk with two or three marked points. If a triangle contains only two marked points, it is said to be \emph{self-folded}. In this case, it contains an arc in its interior called the \emph{self-folded edge}. The boundary of a self-folded triangle is called the \emph{encircling edge}.

If $\phi$ is a complete, saddle-free GMN differential on a compact Riemann surface~$S$, then $\phi$ determines an ideal triangulation of the associated marked bordered surface $(\mathbb{S},\mathbb{M})$. Indeed, $\phi$ determines a horizontal strip decomposition of the underlying Riemann surface~$S$, and we can choose a single generic trajectory within each of the horizontal strips of this decomposition. After passing to the real blow up~$\mathbb{S}$, these trajectories become arcs of an ideal triangulation. This ideal triangulation is known as the \emph{WKB triangulation} for~$\phi$.

In general, if $(\mathbb{S},\mathbb{M})$ is a marked bordered surface, then a \emph{signing} for $(\mathbb{S},\mathbb{M})$ is a function $\epsilon:\mathbb{P}\rightarrow\{\pm1\}$ associating a sign $\epsilon(p)=\pm1$ to every puncture $p\in\mathbb{P}\subset\mathbb{M}$. A \emph{signed triangulation} of~$(\mathbb{S},\mathbb{M})$ is a pair $(T,\epsilon)$ consisting of an ideal triangulation $T$ and a signing~$\epsilon$ of~$(\mathbb{S},\mathbb{M})$. Let us define the \emph{valency} of a puncture $p\in\mathbb{P}$ to be the number of half arcs of~$T$ that are incident to~$p$. Then two signed triangulations $(T_1,\epsilon_1)$ and $(T_2,\epsilon_2)$ are considered to be equivalent if $T_1=T_2$ and the signings $\epsilon_1$ and~$\epsilon_2$ differ only at punctures of valency one. An equivalence class of signed triangulations is called a \emph{tagged triangulation}. (Note that this is closely related to Gaiotto, Moore, and Neitzke's notion of a pop, which is an operation acting on decorated triangulations~\cite{GMN}.)

Suppose $T$ is the WKB triangulation for a complete, saddle-free differential~$\phi\in\mathscr{Q}^\pm(\mathbb{S},\mathbb{M})$. If $p$ is a pole of order two with real residue, then $p$ forms one boundary component of a cylinder of horizontal trajectories. But then the other boundary component consists of saddle trajectories, contradicting the assumption that $\phi$ is saddle-free. It follows that the residue at a pole of order two cannot be real. The choice of the point $\phi\in\mathscr{Q}^\pm(\mathbb{S},\mathbb{M})$ includes a choice of sign for the residue for each pole $p$ of order two, and we can choose the sign $\epsilon(p)\in\{\pm1\}$ so that 
\[
\epsilon(p)\cdot\Res_p(\phi)\in\mathbb{H}
\]
where $\mathbb{H}\subset\mathbb{C}$ is the upper half plane.

In this way, we associate a tagged triangulation to any complete, saddle-free differential $\phi\in\mathscr{Q}^\pm(\mathbb{S},\mathbb{M})$. Conversely, if $\tau$ is an arbitrary tagged triangulation of~$(\mathbb{S},\mathbb{M})$, then there is an associated subset $\mathcal{C}_\tau\subset\mathscr{Q}^\pm(\mathbb{S},\mathbb{M})$ consisting of all points $\phi$ with associated tagged triangulation~$\tau$. These subsets $\mathcal{C}_\tau$ are open (see Sections~5.2 and~10.3 of~\cite{BridgelandSmith}) and their union $\coprod_\tau\mathcal{C}_\tau\subset\mathscr{Q}^\pm(\mathbb{S},\mathbb{M})$ is the subset of all complete saddle-free differentials.

\section{Spin structures and the twisted torus}
\label{sec:SpinStructuresAndTheTwistedTorus}

In this section, we discuss spin structures arising from GMN differentials. The idea of the construction comes from~\cite{KontsevichZorich}. We discuss a related algebraic variety called the twisted torus, which appears in our formulation of the wall-crossing formula.

\subsection{Spin structures}

We begin by recalling the topological definition of a spin structure on a Riemann surface, following~\cite{Atiyah, KontsevichZorich}. Given a Riemann surface~$\Sigma$, let $P_\Sigma\rightarrow\Sigma$ be the bundle whose fiber over $p\in\Sigma$ is the circle of nonzero tangent directions at~$p$. Then a \emph{spin structure} on~$\Sigma$ is a double cover $Q\rightarrow P_\Sigma$ whose restriction to any fiber of~$P_\Sigma$ is the standard double cover $S^1\rightarrow S^1$.

As shown in Section~3 of~\cite{Atiyah}, a line bundle $L$ on~$\Sigma$ satisfying $L\otimes L\cong\omega_\Sigma$ determines a corresponding spin structure on~$\Sigma$. Indeed, consider the map of line bundles 
\begin{equation}
\label{eqn:doublecover}
L^{-1}\rightarrow L^{-1}\otimes L^{-1}=\omega_{\Sigma}^{-1}
\end{equation}
taking a section $s$ of~$L^{-1}$ to the tensor product $s\otimes s$. Note that the sections $s$ and $-s$ have the same image so that \eqref{eqn:doublecover} induces a double cover $Q\rightarrow P_\Sigma$. It restricts to the standard double cover $S^1\rightarrow S^1$ on the fibers of~$P_\Sigma$ and therefore defines a spin structure on~$\Sigma$.

Now suppose we are given a GMN differential~$\phi$ on a compact Riemann surface~$S$. The differential $\phi$ determines the canonical double cover $\Sigma_\phi\rightarrow S$, and there is a canonical meromorphic 1-form~$\lambda$ on~$\Sigma_\phi$ which is holomorphic on the punctured surface~$\Sigma_\phi^\circ$. This 1-form $\lambda$ can be regarded as a global section of the line bundle $\omega_{\Sigma_\phi^\circ}$. Therefore, if we write $D$ for the divisor of zeros of~$\lambda$, then $\omega_{\Sigma_\phi^\circ}$ is isomorphic to the line bundle~$[D]$ associated to~$D$. By Lemma~\ref{lem:doublezeros}, we can write 
\[
D=2p_1+\dots+2p_s
\]
where $p_1,\dots,p_s$ are the zeros of~$\lambda$. We consider the line bundle 
\begin{equation}
\label{eqn:thetacharacteristic}
L\coloneqq[p_1+\dots+p_s],
\end{equation}
which has the property $L\otimes L\cong\omega_{\Sigma_\phi^\circ}$. Thus we have a canonical square root of the holomorphic cotangent bundle, and by the remarks above, there is an associated spin structure on~$\Sigma_\phi^\circ$.

\subsection{Associated quadratic forms}
\label{sec:AssociatedQuadraticForms}

If $Q\rightarrow P_\Sigma$ is a spin structure for the Riemann surface~$\Sigma$, then the group of deck transformations for the covering space $Q$ is isomorphic to~$\mathbb{Z}_2$. Thus we have a group homomorphism $\pi_1(P_\Sigma)\rightarrow\mathbb{Z}_2$, and this homomorphism factors through $H_1(P_\Sigma;\mathbb{Z}_2)$. In this way, we see that spin structures on~$\Sigma$ correspond bijectively to linear maps $H_1(P_\Sigma;\mathbb{Z}_2)\rightarrow\mathbb{Z}_2$ which are nonzero on the cycle represented by the $S^1$ fiber of~$P_\Sigma$. In particular, spin structures can be viewed as elements of $H^1(P_\Sigma;\mathbb{Z}_2)$.

Note that if $\gamma$ is a smooth oriented simple closed curve on~$\Sigma$, then by framing~$\gamma$ with the unit tangent vector field we get a corresponding curve $\vec{\gamma}$ in~$P_\Sigma$ which projects to~$\gamma$. If $\beta$ is the oriented curve obtained from $\gamma$ by reversing the orientation, then the corresponding curve $\vec{\beta}$ is in fact homotopic to~$\vec{\gamma}$ with its orientation reversed; the homotopy is given by simultaneously rotating all framing vectors by an angle~$\pi$. Thus the mod~2 homology class of~$\vec{\gamma}$ in~$P_\Sigma$ is independent of the chosen orientation.

Now suppose that $\gamma=\sum_{i=1}^m\gamma_i$ is a chain representing a class in $H_1(\Sigma;\mathbb{Z}_2)$ where each $\gamma_i$ is a smooth simple closed curve on~$\Sigma$. Choosing an orientation for each of these curves~$\gamma_i$, we obtain curves~$\vec{\gamma}_i$ in~$P_\Sigma$. Let $z$ denote the cycle in~$P_\Sigma$ given by the $S^1$ fiber. Then 
\[
\tilde{\gamma}=\sum_{i=1}^m\vec{\gamma}_i+mz
\]
represents an element of $H_1(P_\Sigma;\mathbb{Z}_2)$. By Theorem~1A in~\cite{Johnson}, this element depends only on the class of~$\gamma$ in $H_1(\Sigma,\mathbb{Z}_2)$ and not on its representation as a sum of smooth simple closed curves or on the choice of orientations for these curves. Therefore we have a map of sets $H_1(\Sigma;\mathbb{Z}_2)\rightarrow H_1(P_\Sigma;\mathbb{Z}_2)$ given by $\gamma\mapsto\tilde{\gamma}$. The following result explains precisely how this map fails to be a group homomorphism.

\begin{proposition}[\cite{Johnson}, Theorem~1B]
\label{prop:nonhomom}
The mapping defined above satisfies 
\[
\widetilde{\gamma_1+\gamma_2}=\tilde{\gamma_1}+\tilde{\gamma_2}+(\gamma_1\cdot\gamma_2)z
\]
where $\gamma_1\cdot\gamma_2$ denotes the intersection pairing of~$\gamma_1$,~$\gamma_2\in H_1(\Sigma;\mathbb{Z}_2)$.
\end{proposition}

By a $\mathbb{Z}_2$-valued \emph{quadratic form} on $H_1(\Sigma;\mathbb{Z}_2)$ with the associated bilinear form $(\gamma_1,\gamma_2)\mapsto\gamma_1\cdot\gamma_2$, we mean any function $q:H_1(\Sigma;\mathbb{Z}_2)\rightarrow\mathbb{Z}_2$ such that 
\begin{equation}
\label{eqn:quadraticform}
q(\gamma_1+\gamma_2)=q(\gamma_1)+q(\gamma_2)+\gamma_1\cdot\gamma_2
\end{equation}
for $\gamma_1$,~$\gamma_2\in H_1(\Sigma;\mathbb{Z}_2)$. It follows from Proposition~\ref{prop:nonhomom} that if $\omega\in H^1(P_\Sigma;\mathbb{Z}_2)$ is a spin structure, then there is an associated quadratic form $q_\omega$ given by 
\[
q_\omega(\gamma)=\omega(\tilde{\gamma})
\]
where the right hand side denotes the pairing of homology and cohomology.

\subsection{Spin structures from differentials}

We have now seen that any spin structure on a Riemann surface gives rise to an associated quadratic form. In particular, if we have a GMN differential~$\phi$, then there is an associated spin structure $\omega$ on the surface $\Sigma_\phi^\circ$, and this gives rise to a quadratic form $q_\omega$ on $H_1(\Sigma_\phi^\circ;\mathbb{Z}_2)$. Given a class $\gamma\in\Gamma_\phi$, let us write $\gamma$ also for the image of this class under the map $\Gamma_\phi\rightarrow H_1(\Sigma_\phi^\circ;\mathbb{Z}_2)$ given by reduction modulo~2. We will be interested in the map 
\[
\xi:\Gamma_\phi\rightarrow\mathbb{C}^*, \quad \xi(\gamma)=(-1)^{q_\omega(\gamma)}.
\]
By the identity~\eqref{eqn:quadraticform}, this map satisfies 
\begin{equation}
\label{eqn:basepoint}
\xi(\gamma_1+\gamma_2)=(-1)^{\langle\gamma_1,\gamma_2\rangle}\xi(\gamma_1)\xi(\gamma_2)
\end{equation}
where $\langle\gamma_1,\gamma_2\rangle$ denotes the intersection pairing of~$\gamma_1$,~$\gamma_2\in\Gamma_\phi$. Below we will give an alternative description of this quadratic form, using ideas from~\cite{KontsevichZorich} (see also~\cite{GMN}, Section~7.7).

Consider the differential $\lambda$ on the Riemann surface~$\Sigma_\phi^\circ$. At any point of $\Sigma_\phi^\circ$ which is not a zero of~$\lambda$, we can find a unique tangent vector $v$ such that $\lambda(v)=1$. These tangent vectors provide a nonvanishing section of the tangent bundle of~$\Sigma_\phi^\circ$ over the complement of the zeros of~$\lambda$, and therefore the tangent bundle is trivial over this set. Let $\gamma$ be a smooth oriented closed curve on~$\Sigma_\phi^\circ$ which does not meet any zero of~$\lambda$. Since every tangent space over~$\gamma$ is canonically identified with~$\mathbb{C}$, we get a map $G:\gamma\rightarrow S^1\subset\mathbb{C}$ sending a point $p$ on~$\gamma$ to the tangent direction determined by the orientation of~$\gamma$ at~$p$. We define the \emph{index} $\ind_\gamma(\lambda)$ to be the integer such that $2\pi\cdot\ind_\gamma(\lambda)$ is the total change in angle between $G(p)$ and $1\in S^1$ as $p$ goes around the curve~$\gamma$.

\begin{proposition}
\label{prop:indexquadraticform}
Let $\omega\in H^1(P_\Sigma;\mathbb{Z}_2)$ be the spin structure on $\Sigma=\Sigma_\phi^\circ$ determined by the GMN~differential~$\phi$. If $\gamma$ is any smooth oriented simple closed curve on~$\Sigma$ which does not meet a zero of~$\lambda$, then 
\[
q_\omega(\gamma)=\ind_\gamma(\lambda)+1\mod2.
\]
\end{proposition}

\begin{proof}
Let $\alpha$ be a loop in $P_\Sigma$ which is obtained by framing $\gamma\subset\Sigma$ so that the framing vector $v$ at any point of $\gamma$ satisfies $\lambda(v)=1$. By construction, the line bundle $L$ defined in~\eqref{eqn:thetacharacteristic} has a section that vanishes precisely on the set of zeros of $\lambda$. Hence we have a section of~$L$ which is nonvanishing on the complement of the zeros, and $L$ is trivial over this complement. The map~\eqref{eqn:doublecover} induces a double cover $Q\rightarrow P_\Sigma$, and the triviality of~$L$ implies that if we lift~$\alpha$ to this double cover, then the endpoints of the lifted curve coincide. In other words, the deck transformation of~$Q$ corresponding to $\alpha\in\pi_1(P_\Sigma)$ is trivial, which implies $\omega(\alpha)=0$ by the discussion in Section~\ref{sec:AssociatedQuadraticForms}.

On the other hand, if $z$ is the loop in~$P_\Sigma$ given by the $S^1$ fiber of $P_\Sigma\rightarrow\Sigma$ then we have $\omega(z)=1\in\mathbb{Z}_2$. Note that the preimage $T_\gamma$ of $\gamma\subset\Sigma$ in~$P_\Sigma$ is topologically a torus, and the loops $\alpha$ and $z$ generate the fundamental group $\pi_1(T_\gamma)$. If $\vec{\gamma}$ is the curve in~$P_\Sigma$ obtained by equipping~$\gamma$ with the field of unit vectors tangent to~$\gamma$, then $\vec{\gamma}$ is equal in~$\pi_1(P_\Sigma)$ to a concatenation of~$\alpha$ with $\ind_\gamma(\lambda)$ copies of the loop~$z$. From this we see that 
\[
\omega(\vec{\gamma})=\ind_\gamma(\lambda)\mod2
\]
which implies the desired result since $\omega(\tilde{\gamma})=\omega(\vec{\gamma})+1$.
\end{proof}

Using Proposition~\ref{prop:indexquadraticform}, we can give an alternative description of the map $\xi$ introduced above. Indeed, if $\gamma=\sum_i\gamma_i$ is a class in~$\Gamma_\phi$ which is a sum of mutually nonintersecting smooth oriented simple closed curves~$\gamma_i$ that do not meet the zeros of~$\lambda$, then we have 
\begin{equation}
\label{eqn:formulaforxi}
\xi(\gamma)=\prod_i(-1)^{\ind_{\gamma_i}(\lambda)+1}
\end{equation}
by Proposition~\ref{prop:indexquadraticform} and the fact that $\langle\gamma_i,\gamma_j\rangle=0$ for $i\neq j$. This formula provides a means of computing the value of the map~$\xi$ on a given class in~$\Gamma_\phi$.

\begin{proposition}
The map $\xi$ satisfies $\xi(\gamma)=-1$ if $\gamma\in\Gamma_\phi$ is the class of a non-closed saddle connection and $\xi(\gamma)=+1$ if $\gamma$ is the class of a closed saddle connection.
\end{proposition}

\begin{proof}
If $\gamma\in\Gamma_\phi$ is the class of a non-closed saddle connection, then $\gamma$ can be represented by a single smooth oriented closed curve in~$\Sigma_\phi^\circ$ whose projection to~$S$ is a small loop surrounding the saddle connection. (See Figure~30 of~\cite{GMN}). It is then straightforward to check that $\xi(\gamma)=-1$ using~\eqref{eqn:formulaforxi}. On the other hand, if $\gamma$ is the class of a closed saddle connection, then it can be represented by a union of two disjoint closed loops in~$\Sigma_\phi^\circ$ which are interchanged by the covering involution $\Sigma_\phi^\circ\rightarrow\Sigma_\phi^\circ$. (See Figure~33 of~\cite{GMN}.) We then have $\xi(\gamma)=+1$ by~\eqref{eqn:formulaforxi}.
\end{proof}

\subsection{The twisted torus}

Let $\phi$ be a GMN differential, and consider the associated lattice $\Gamma_\phi\cong\mathbb{Z}^n$ equipped with the skew form $\langle-,-\rangle$. This lattice defines an algebraic torus 
\[
\mathbb{T}_+=\Hom_{\mathbb{Z}}(\Gamma_\phi,\mathbb{C}^*)\cong(\mathbb{C}^*)^n.
\]
We will be interested in a related object 
\[
\mathbb{T}_-=\left\{g:\Gamma_\phi\rightarrow\mathbb{C}^*:g(\gamma_1+\gamma_2)=(-1)^{\langle\gamma_1,\gamma_2\rangle}g(\gamma_1)g(\gamma_2)\right\}
\]
known as the \emph{twisted torus}~\cite{Bridgeland19}.

There is an action of $\mathbb{T}_+$ on~$\mathbb{T}_-$ given by 
\[
(f\cdot g)(\gamma)=f(\gamma)g(\gamma)\in\mathbb{C}^*
\]
for $f\in\mathbb{T}_+$ and $g\in\mathbb{T}_-$, and this action is free and transitive. Thus, after choosing a basepoint in the twisted torus~$\mathbb{T}_-$, we get an identification of~$\mathbb{T}_-$ and~$\mathbb{T}_+$. This identification gives $\mathbb{T}_-$ the structure of an algebraic variety, and this variety structure is independent of the choice of basepoint since the translation maps on~$\mathbb{T}_+$ are algebraic. The coordinate ring $\mathbb{C}[\mathbb{T}_-]$ of the twisted torus is spanned as a vector space by the functions 
\[
x_\gamma:\mathbb{T}_-\rightarrow\mathbb{C}^*, \quad x_\gamma(g)=g(\gamma)\in\mathbb{C}^*,
\]
which we call the \emph{twisted characters}. The twisted torus $\mathbb{T}_-$ has a natural Poisson structure with the Poisson bracket given on the twisted characters by 
\[
\{x_\alpha,x_\beta\}=\langle\alpha,\beta\rangle\cdot x_\alpha\cdot x_\beta
\]
for $\alpha$,~$\beta\in\Gamma_\phi$.

For any GMN differential~$\phi$, we have constructed an associated map $\xi:\Gamma_\phi\rightarrow\mathbb{C}^*$. This map satisfies the relation~\eqref{eqn:basepoint} and can therefore be considered as a point in the twisted torus~$\mathbb{T}_-$. In the following, we will use~$\xi$ as a canonical basepoint to identify~$\mathbb{T}_-$ and~$\mathbb{T}_+$. When there is no possibility of confusion, we will denote the twisted torus simply by~$\mathbb{T}$.

\section{The wall-crossing formula for quadratic differentials}
\label{sec:TheWallCrossingFormulaForQuadraticDifferentials}

In this section, we formulate the Kontsevich-Soibelman wall-crossing formula as a statement about finite-length trajectories of quadratic differentials.

\subsection{The ray diagram}

Fix a GMN differential $\phi$. When we talk about a \emph{ray} in~$\mathbb{C}^*$, we will always mean a subset of the form $\ell=\mathbb{R}_{>0}\cdot e^{i\pi\theta}$ for some $\theta\in\mathbb{R}$. The ray $\ell$ is said to be \emph{active} if there exists a finite-length trajectory for~$\phi$ with phase~$\theta$.

The \emph{ray diagram} associated to the differential~$\phi$ is defined as the union of all active rays in~$\mathbb{C}^*$. An example is illustrated in Figure~\ref{fig:raydiagram}. Note that the phase of a finite-length trajectory is well defined up to addition of an integer, and therefore if $\ell$ is a ray in the ray diagram, then so is $-\ell$.

\begin{figure}[ht]
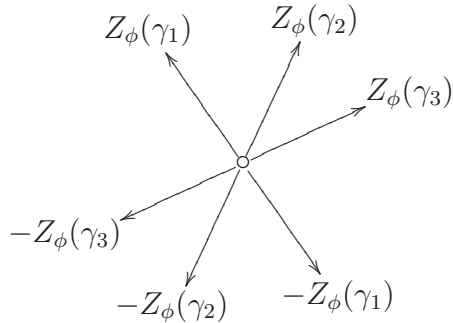

\begin{center}
\[
\xy /l1.5pc/:
{\xypolygon18"A"{~:{(2,2):}~>{}}};
(1,0)*{\circ}="a";
(3,-2.8)*{Z_\phi(\gamma_1)};
(-1,2.8)*{-Z_\phi(\gamma_1)};
(-0.5,-3)*{Z_\phi(\gamma_2)};
(2.5,3)*{-Z_\phi(\gamma_2)};
(-2.5,-1.5)*{Z_\phi(\gamma_3)};
(4.75,1.5)*{-Z_\phi(\gamma_3)};
\xygraph{
"a":"A2",
"a":"A11",
"a":"A9",
"a":"A18",
"a":"A5",
"a":"A14",
}
\endxy
\]
\end{center}
\caption{A ray diagram.\label{fig:raydiagram}}
\end{figure}

Suppose now that the differential $\phi$ is generic. Then the \emph{height} of a ray $\ell\subset\mathbb{C}^*$ is defined to be the number 
\[
H(\ell)=\inf\{|Z_\phi(\gamma)|:\text{$\gamma\in\Gamma_\phi$ such that $Z_\phi(\gamma)\in\ell$ and $\Omega_\phi(\gamma)\neq0$}\},
\]
whenever the set on the right hand side of this expression is non-empty. Otherwise, $\ell$ is considered to have infinite height. It follows from the remarks in~\cite{Bridgeland19}, Section~2.5, that for any $H>0$ one has at most finitely many rays of height $<H$.

\subsection{BPS automorphisms}
\label{sec:BPSautomorphisms}

If $\phi$ is a generic GMN differential, then the \emph{Donaldson-Thomas invariant} for $\gamma\in\Gamma$ is defined by the formula 
\[
DT_\phi(\gamma)=\sum_{\gamma=m\alpha}\frac{1}{m^2}\Omega_\phi(\alpha)\in\mathbb{Q}
\]
where the sum is over all integers $m>0$ such that $\gamma$ is divisible by~$m$ in the lattice~$\Gamma$. By the M\"obius inversion formula, one can express the numbers $\Omega_\phi(\gamma)$ in terms of the $DT_\phi(\gamma)$, and so the BPS and Donaldson-Thomas invariants are equivalent data.

Given any ray $\ell\subset\mathbb{C}^*$, we can consider the associated formal generating series 
\begin{equation}
\label{eqn:generatingseries}
DT_\phi(\ell)=\sum_{Z_\phi(\gamma)\in\ell}DT_\phi(\gamma)\cdot x_\gamma
\end{equation}
for Donaldson-Thomas invariants. We would like to view this generating series as a well defined holomorphic function on the twisted torus $\mathbb{T}$. To do this, we must specify the a suitable domain in~$\mathbb{T}$. For any acute sector $\Delta\subset\mathbb{C}^*$ and real number $R>0$, we consider the open set $U_\Delta(R)\subset\mathbb{T}$ defined as the interior of 
\[
\left\{g\in\mathbb{T}:Z_\phi(\gamma)\in\Delta\text{ and }\Omega_\phi(\gamma)\neq0\implies|g(\gamma)|<\exp(-R\|\gamma\|)\right\}\subset\mathbb{T}.
\]
It is nonempty by~Lemma~B.2 of~\cite{Bridgeland19}. Note that for each $\gamma\in\Gamma_\phi$, the corresponding BPS~invariant satisfies $|\Omega_\phi(\gamma)|\leq2$ by Lemma~5.1 of~\cite{BridgelandSmith}. Hence for sufficiently large $R>0$, we have
\[
\sum_{\gamma\in\Gamma}|\Omega_\phi(\gamma)|\cdot e^{-R|Z_\phi(\gamma)|}<\infty.
\]
This shows that the data $(\Gamma_\phi,Z_\phi,\Omega_\phi)$ define a convergent BPS structure in the sense of~\cite{Bridgeland19}. Hence we have the following.

\begin{proposition}[\cite{Bridgeland19}, Proposition~4.1]
\label{prop:BPSautomorphism}
Let $\Delta\subset\mathbb{C}^*$ be a convex sector. Then for sufficiently large $R>0$, the following statements hold: 
\begin{enumerate}
\item For each ray $\ell\subset\Delta$, the power series~\eqref{eqn:generatingseries} is absolutely convergent on $U_\Delta(R)$ and thus defines a holomorphic function 
\[
DT_\phi(\ell):U_\Delta(R)\rightarrow\mathbb{C}.
\]
\item The time-1 Hamiltonian flow $\exp\{DT_\phi(\ell),-\}$ of the function $DT_\phi(\ell)$ defines a holomorphic embedding 
\[
\mathbf{S}_\phi(\ell):U_\Delta(R)\rightarrow\mathbb{T}.
\]
\item For every $H>0$, the composition
\[
\mathbf{S}_{\phi,<H}(\Delta)=\mathbf{S}_\phi(\ell_1)\circ\mathbf{S}_\phi(\ell_2)\circ\dots\circ\mathbf{S}_\phi(\ell_k)
\]
exists where $\ell_1,\ell_2,\dots,\ell_k\subset\Delta$ are the rays of height $<H$ in the sector $\Delta$ in the clockwise order, and the pointwise limit 
\[
\mathbf{S}_\phi(\Delta)=\lim_{H\rightarrow\infty}\mathbf{S}_{\phi,<H}(\Delta):U_\Delta(R)\rightarrow\mathbb{T}
\]
is a well defined holomorphic embedding.
\end{enumerate}
\end{proposition}

We think of the map $\mathbf{S}_\phi(\Delta)$ defined by Proposition~\ref{prop:BPSautomorphism} as a partially defined automorphism of the twisted torus and call it the \emph{BPS automorphism} associated to the sector~$\Delta$. We note that similar analytic maps have been studied by Kontsevich and Soibelman~\cite{KontsevichSoibelman4}, who described a general framework for studying wall-crossing formulas in an analytic context.

For a generic GMN~differential~$\phi$, the data $(\Gamma_\phi,Z_\phi,\Omega_\phi)$ form a generic, integral, and ray-finite BPS structure in the sense of~\cite{Bridgeland19}. Using this fact, we can give a more explicit description of the maps $\mathbf{S}_\phi(\ell)$.

\begin{proposition}[\cite{Bridgeland19}, Proposition~B.6]
\label{prop:KSautomorphism}
If $\phi$ is a generic GMN differential, then for any ray $\ell\subset\mathbb{C}^*$, the holomorphic embedding $\mathbf{S}_\phi(\ell)$ extends to a birational automorphism of~$\mathbb{T}$, whose action on the twisted characters is given by 
\[
\mathbf{S}_\phi(\ell)^*(x_\beta)=x_\beta\cdot\prod_{Z_\phi(\gamma)\in\ell}(1-x_\gamma)^{\Omega_\phi(\gamma)\cdot\langle\beta,\gamma\rangle}.
\]
\end{proposition}

\subsection{The wall-crossing formula}

We will now study the behavior of the BPS automorphisms $\mathbf{S}_\phi(\Delta)$ as we vary the differential~$\phi$ in the moduli space $\mathscr{Q}^\pm(\mathbb{S},\mathbb{M})$ for a fixed marked bordered surface $(\mathbb{S},\mathbb{M})$. Note that if $\tau$ is any tagged triangulation of~$(\mathbb{S},\mathbb{M})$, then the canonical double covers $\Sigma_\phi$ define a family of Riemann surfaces over the open set~$\mathcal{C}_\tau\subset\mathscr{Q}(\mathbb{S},\mathbb{M})$. It follows that the lattices $\Gamma_\phi$ form a local system over~$\mathcal{C}_\tau$ with flat connection given by the Gauss-Manin connection. Using this flat connection, we can identify the $\Gamma_\phi$ for $\phi\in\mathcal{C}_\tau$ with a single lattice. In particular, the associated twisted torus $\mathbb{T}$ is independent of $\phi\in\mathcal{C}_\tau$.

In Appendix~\ref{sec:TheMotivicWallCrossingFormula}, we derive the following statement from the motivic wall-crossing formula. It describes implicitly how the BPS invariants $\Omega_\phi(\gamma)$ jump as $\phi$ varies.

\begin{theorem}
\label{thm:firstWCF}
Let $\tau$ be a tagged triangulation of a marked bordered surface $(\mathbb{S},\mathbb{M})$, and let $\Delta$ be a sector contained in the upper half plane. Suppose $\phi_t$,~$t\in[0,1]$, is a path in~$\mathcal{C}_\tau\subset\mathscr{Q}^\pm(\mathbb{S},\mathbb{M})$ with generic endpoints such that the boundary rays of~$\Delta$ are non-active for each differential $\phi_t$. Then 
\[
\mathbf{S}_{\phi_0}(\Delta)=\mathbf{S}_{\phi_1}(\Delta).
\]
\end{theorem}

While this is the version of the wall-crossing formula that follows most readily from the one stated in the references~\cite{Meinhardt, DavisonMeinhardt1}, our main result will ultimately allow us to drop the requirement that the path $\phi_t$ lies entirely in the domain~$\mathcal{C}_\tau$.

\section{Birationality of BPS automorphisms}

In this section, we prove that the BPS automorphisms $\mathbf{S}_\phi(\Delta)$ introduced in the last section extend to birational automorphisms of the twisted torus~$\mathbb{T}$. We do this by relating the BPS automorphisms to Fock-Goncharov coordinates on moduli spaces of flat $\PGL_2(\mathbb{C})$-connections.

\subsection{Fock-Goncharov coordinates}

The birational transformations that we consider in this section arose from the work of Fock and Goncharov on an algebro-geometric approach to higher Teichm\"uller theory~\cite{FockGoncharov1}. For any marked bordered surface $(\mathbb{S},\mathbb{M})$, Fock and Goncharov defined a moduli space denoted $\mathscr{X}(\mathbb{S},\mathbb{M})$. This moduli space parametrizes flat $\PGL_2(\mathbb{C})$-connections on the punctured surface $\mathbb{S}\setminus\mathbb{M}$ with additional data associated to the marked points. We refer to~\cite{Allegretti19b} for the details of this construction. What is important for us is that this moduli space $\mathscr{X}(\mathbb{S},\mathbb{M})$ has an interesting atlas of coordinate charts, which we will now describe.

Fix a marked bordered surface $(\mathbb{S},\mathbb{M})$, and assume this surface admits an ideal triangulation with $n>1$ arcs. By a \emph{signed arc} of~$(\mathbb{S},\mathbb{M})$ we will mean a triple $(T,\epsilon,j)$ where $(T,\epsilon)$ is a signed triangulation of~$(\mathbb{S},\mathbb{M})$ and $j$ is an arc of~$T$. If $\tau$ is a tagged triangulation of~$(\mathbb{S},\mathbb{M})$, let $\mathcal{A}_\tau$ be the set of all signed arcs $(T,\epsilon,j)$ such that $(T,\epsilon)$ is a representative for~$\tau$. We define an equivalence relation on this set $\mathcal{A}_\tau$ as follows. If $(T,\epsilon)$ and $(T,\epsilon')$ are two signed triangulations representing~$\tau$ where the signing $\epsilon'$ differs from $\epsilon$ at a single puncture $p$ of valency one with respect to~$T$, then $p$ is incident to a unique arc $j$ which is the interior edge of a self-folded triangle with encircling edge~$j'$. In this case we say that the tagged arc $(T,\epsilon',j')$ is equivalent to the tagged arc $(T,\epsilon,j)$. This generates the desired equivalence relation on~$\mathcal{A}_\tau$, and we can define a tagged arc of~$\tau$ to be an equivalence class in~$\mathcal{A}_\tau$. We typically represent a tagged triangulation $\tau$ by a fixed signed triangulation $(T,\epsilon)$ and use the same notation for an arc of~$T$ and the corresponding tagged arc of~$\tau$ of~$(\mathbb{S},\mathbb{M})$. For each tagged triangulation~$\tau$, we write $\Gamma_\tau\cong\mathbb{Z}^n$ for a lattice with basis $\{\gamma_j\}$ indexed by tagged arcs~$j$ of~$\tau$.

The lattice $\Gamma_\tau$ comes with a natural skew-form $\langle-,-\rangle:\Gamma_\tau\times\Gamma_\tau\rightarrow\mathbb{Z}$. To define it, let $(T,\epsilon)$ be a signed triangulation representing~$\tau$. If $j$ is any arc of~$T$, we consider the arc $\pi_T(j)$ defined as follows: If $j$ is the interior edge of a self-folded triangle, we let $\pi_T(j)$ be the encircling edge, and we let $\pi_T(j)=j$ otherwise. For each non-self-folded triangle $t$ of~$T$, we define a number denoted $b_{ij}^t$ by the following rules: 
\begin{enumerate}
\item $b_{ij}^t=+1$ if $\pi_T(i)$ and $\pi_T(j)$ are edges of~$t$ with $\pi_T(j)$ following $\pi_T(i)$ in the clockwise order defined by the orientation.
\item $b_{ij}^t=-1$ if the same holds with the counterclockwise order.
\item $b_{ij}^t=0$ otherwise.
\end{enumerate}
Finally, if $i$ and $j$ are tagged arcs of~$\tau$, we define 
\[
\langle\gamma_j,\gamma_i\rangle=\sum_tb_{ij}^t
\]
where the sum is over all non-self-folded triangles in~$T$. We extend this to a form on~$\Gamma_\tau$ by bilinearity.

It follows from the work of Fock and Goncharov~\cite{FockGoncharov1} that there exists a birational map 
\[
X_\tau:\mathscr{X}(\mathbb{S},\mathbb{M})\dashrightarrow\mathbb{T}_\tau\coloneqq\Hom_{\mathbb{Z}}(\Gamma_\tau,\mathbb{C}^*)
\]
from the moduli space described above to an algebraic torus $\mathbb{T}_\tau\cong(\mathbb{C}^*)^n$. The components of this map corresponding to tagged arcs of~$\tau$ are called \emph{Fock-Goncharov coordinates}. We refer to~\cite{Allegretti19b} for the detailed construction of this map~$X_\tau$. Here we are mainly interested in the relationship between the Fock-Goncharov coordinates associated to different tagged triangulations.

Suppose $T$ is an ideal triangulation of~$(\mathbb{S},\mathbb{M})$ and $k$ is an arc of~$T$. We say that an ideal triangulation $T'$ is obtained from~$T$ by a \emph{flip} of~$k$ if $T'\neq T$ and there is an arc~$k'$ of~$T'$ such that $T\setminus\{k\}=T'\setminus\{k'\}$. Similarly, suppose $\tau$ is a tagged triangulation and $k$ is a tagged arc of~$\tau$. In this case, we say that a tagged triangulation $\tau'$ is obtained from~$\tau$ by a flip of~$k$ if $\tau$ and $\tau'$ are represented by signed triangulations $(T,\epsilon)$ and $(T',\epsilon')$, respectively, and $T'$ is obtained from $T$ by a flip of~$k$. Note that while we cannot flip a self-folded edge in an ordinary ideal triangulation, we can flip any tagged arc of a tagged triangulation.

Let $\tau$ be a tagged triangulation and $\tau'$ the tagged triangulation obtained from~$\tau$ by flipping the tagged arc~$k$. Then the transition map $\mu_k=X_{\tau'}\circ X_\tau^{-1}$ is a birational map 
\[
\mu_k:\mathbb{T}_\tau\dashrightarrow\mathbb{T}_{\tau'}
\]
which can be described explicitly. In the following proposition, we will write $X_\gamma(f)\coloneqq f(\gamma)$ for $\gamma\in\Gamma_\tau$ and $f\in\mathbb{T}_\tau$, and write $X_\gamma'(f)\coloneqq f(\gamma)$ for $\gamma\in\Gamma_{\tau'}$ and $f\in\mathbb{T}_{\tau'}$. We will use the same notation for a tagged arc in~$\tau$ and the corresponding tagged arc in the flipped triangulation~$\tau'$.

\begin{proposition}[\cite{Allegretti19b}, Section~4 and \cite{FockGoncharov2}, Section~2.1]
\label{prop:clustertransformation}
The rational map $\mu_k$ can be written as a composition $\mu_k=\iota_k\circ\kappa_k$ where 
\begin{enumerate}
\item $\iota_k$ is the isomorphism $\mathbb{T}_\tau\rightarrow\mathbb{T}_{\tau'}$ given by 
\[
\iota_k^*(X_{\gamma_j}')=
\begin{cases}
X_{\gamma_k}^{-1} & \text{if $j=k$} \\
X_{\gamma_j}X_{\gamma_k}^{[\langle\gamma_k,\gamma_j\rangle]_+} & \text{if $j\neq k$}
\end{cases}
\]
where $[n]_+\coloneqq\max(n,0)$.
\item $\kappa_k$ is the birational automorphism $\mathbb{T}_\tau\dashrightarrow\mathbb{T}_\tau$ given by 
\[
\kappa_k^*(X_\gamma)=X_\gamma\cdot(1+X_{\gamma_k})^{\langle\gamma,\gamma_k\rangle}.
\]
\end{enumerate}
\end{proposition}

If $(\mathbb{S},\mathbb{M})$ is not a closed surface with exactly one puncture, then any two tagged triangulations of~$(\mathbb{S},\mathbb{M})$ are related by a sequence of flips (\cite{FST}, Proposition~7.10). Thus Proposition~\ref{prop:clustertransformation} can be used to calculate the transition map $X_{\tau'}\circ X_\tau^{-1}$ for any tagged triangulations~$\tau$ and~$\tau'$ in this case.

\subsection{The main result}
\label{sec:TheMainResult}

Consider a complete, generic GMN differential~$\phi$. We further assume that the associated marked bordered surface~$(\mathbb{S},\mathbb{M})$ is not a closed surface with exactly one puncture. If $\Delta\subset\mathbb{C}^*$ is a convex sector whose boundary rays are non-active with phases~$\theta_1$ and~$\theta_2$, then each rotated differential $\phi_i=e^{-2i\theta_i}\cdot\phi$ is complete and saddle-free and hence determines a tagged WKB~triangulation~$\tau_i$. As we have seen, the Fock-Goncharov coordinates provide a birational map 
\begin{equation}
\label{eqn:twocharts}
X_{\tau_i}:\mathscr{X}(\mathbb{S},\mathbb{M})\dashrightarrow\Hom_{\mathbb{Z}}(\Gamma_{\tau_i},\mathbb{C}^*)
\end{equation}
where $\Gamma_{\tau_i}\cong\mathbb{Z}^n$ is the lattice spanned by the set of tagged arcs of~$\tau_i$. By Lemma~10.3 of~\cite{BridgelandSmith}, the lattice $\Gamma_{\tau_i}$ is canonically isomorphic to~$\Gamma_{\phi_i}$. There is a family of Riemann surfaces over~$\mathbb{R}$ where the fiber over $\theta\in\mathbb{R}$ is the canonical double cover for the rotated differential $\phi_\theta=e^{-2i\theta}\cdot\phi$. The homology groups of these Riemann surfaces form a local system of lattices over~$\mathbb{R}$ with flat connection given by the Gauss-Manin connection. Using this flat connection, we can identify the lattices $\Gamma_{\phi_i}$ with $\Gamma_\phi$. We can therefore think of the maps~\eqref{eqn:twocharts} as taking values in the torus $\mathbb{T}_+=\Hom_{\mathbb{Z}}(\Gamma_\phi,\mathbb{C}^*)$. The following is the main result of this paper.

\begin{theorem}
\label{thm:main}
Let $\phi$ be a complete, generic GMN differential, and assume that the associated marked bordered surface~$(\mathbb{S},\mathbb{M})$ is not a closed surface with exactly one puncture. Let $\Delta\subset\mathbb{C}^*$ be a convex sector whose boundary rays are non-active with phases~$\theta_1$ and~$\theta_2$. Then 
\begin{enumerate}
\item There is a distinguished basepoint $\xi\in\mathbb{T}_-$ such that $\xi(\gamma)=-1$ if $\gamma\in\Gamma_\phi$ is the class of a non-closed saddle connection and $\xi(\gamma)=+1$ if $\gamma$ is the class of a closed saddle connection.
\item $\mathbf{S}_\phi(\Delta)$ extends to a birational automorphism of~$\mathbb{T}_-$. If we use the basepoint $\xi$ to identify~$\mathbb{T}_-$ with~$\mathbb{T}_+$, then this is the birational automorphism of~$\mathbb{T}_+$ relating the maps 
\begin{equation}
\label{eqn:sametarget}
X_{\tau_i}:\mathscr{X}(\mathbb{S},\mathbb{M})\dashrightarrow\mathbb{T}_+
\end{equation}
where $\tau_i$ is the tagged WKB triangulation determined by the rotated differential $\phi_i=e^{-2i\theta_i}\cdot\phi$.
\end{enumerate}
\end{theorem}

Theorem~\ref{thm:main} gives a way of computing the BPS automorphism $\mathbf{S}_\phi(\Delta)$ for a general sector $\Delta\subset\mathbb{C}^*$ and implies Theorem~\ref{thm:introWCF} from the introduction. Note that we have already proved part~(1) in Section~\ref{sec:SpinStructuresAndTheTwistedTorus}. We will devote the remainder of the present section to the proof of part~(2).

\subsection{Stratification of moduli spaces}

We begin by describing some topological aspects of the moduli space of GMN differentials. Given a complete GMN differential $\phi\in\mathscr{Q}^\pm(\mathbb{S},\mathbb{M})$, let us write $r_\phi$ for the number of recurrent trajectories which approach a zero at one end and $s_\phi$ for the number of horizontal saddle connections. Following Section~5.2 of~\cite{BridgelandSmith}, we consider for every integer $p\geq0$ the subset 
\[
B_p=\{\phi\in\mathscr{Q}^\pm(\mathbb{S},\mathbb{M}):\text{$\phi$ is complete and $r_\phi+2s_\phi\leq p$}\}.
\]
Thus $B_0=B_1$ is the set of complete saddle-free differentials, and $B_2$ is the set of complete differentials having at most one horizontal saddle connection. As in~\cite{BridgelandSmith}, we define $F_0=B_0$ and $F_p=B_p\setminus B_{p-1}$ for $p\geq1$. Each of these sets $F_p$ is locally closed in $\mathscr{Q}^\pm(\mathbb{S},\mathbb{M})$ and the union $\coprod_{p=0}^\infty F_p\subset\mathscr{Q}^\pm(\mathbb{S},\mathbb{M})$ is the set of all complete GMN~differentials. Thus, the $F_p$ provide a stratification of this set.

If $p\geq2$, then by Proposition~5.5 of~\cite{BridgelandSmith}, the set $F_p$ has codimension one in~$B_p$, and hence one can think of~$F_p$ as a wall in~$B_p$, potentially separating two connected components of~$B_{p-1}$. A crucial feature of the stratification by the sets $F_p$ is the ``walls have ends'' property proved in~\cite{BridgelandSmith}, Section~5.6:

\begin{proposition}
\label{prop:wallshaveends}
Assume $(\mathbb{S},\mathbb{M})$ is not a closed surface with exactly one puncture, and take $p>2$. Let $C$ be any connected component of $F_p\subset\mathscr{Q}^\pm(\mathbb{S},\mathbb{M})$. Then there is a point $\phi$ in the closure of~$C$ such that for any neighborhood $\phi\in U\subset B_p$, we can find a smaller neighborhood $\phi\in V\subset U$ such that $V\cap B_{p-1}$ is connected.
\end{proposition}

\subsection{Completing the proof}

In what follows, we will assume $(\mathbb{S},\mathbb{M})$ is not a closed surface with exactly one puncture, and we take $\phi\in\mathscr{Q}^\pm(\mathbb{S},\mathbb{M})$. We will say that the differential $\phi$ is \emph{good} if there exists $\varepsilon>0$ such that if $-\varepsilon<\theta_1<0<\theta_2<\varepsilon$ and the rotated differentials $e^{-2i\theta_j}\cdot\phi$ are saddle-free, then the conclusion of part~(2) of Theorem~\ref{thm:main} holds for the sector $\Delta\subset\mathbb{C}^*$ having boundary rays $\mathbb{R}_{>0}\cdot e^{i\theta_j}$.

\begin{lemma}
\label{lem:F0}
If $\phi$ is generic differential in~$F_0$ then $\phi$ is good.
\end{lemma}

\begin{proof}
If $\phi\in F_0$ then $\phi$ is a saddle-free differential and we can find $\varepsilon>0$ such that if $-\varepsilon<\theta_1<0<\theta_2<\varepsilon$ then the rotated differentials $e^{-2i\theta_j}\cdot\phi$ are saddle-free and have the same associated tagged triangulation as~$\phi$. Therefore the transformation relating the maps~\eqref{eqn:sametarget} is the identity. If $\Delta\subset\mathbb{C}^*$ is the sector having boundary rays $\mathbb{R}_{>0}\cdot e^{i\theta_j}$, then $\Delta$ contains no active rays, so for generic~$\phi$, the BPS~automorphism $\mathbf{S}_\phi(\Delta)$ is defined and equals the identity. Hence $\phi$ is good.
\end{proof}

\begin{lemma}
\label{lem:F1}
If $\phi$ is a generic differential in~$F_2$ then $\phi$ is good.
\end{lemma}

\begin{proof}
If $\phi\in F_2$ then $\phi$ has a unique horizontal saddle connection. Let us assume for the time being that this is not a closed saddle connection. By Proposition~5.5 of~\cite{BridgelandSmith}, there exists $\varepsilon>0$ such that if $-\varepsilon<\theta_1<0<\theta_2<\varepsilon$ then the rotated differentials $e^{-2i\theta_j}\cdot\phi$ are saddle-free and the associated tagged triangulations are related by a flip. Let $\Delta\subset\mathbb{C}^*$ be the sector having boundary rays $\mathbb{R}_{>0}\cdot e^{i\theta_j}$. If $\phi$ is generic, then by Proposition~\ref{prop:KSautomorphism} the associated BPS~automorphism is given by 
\[
\mathbf{S}_\phi(\Delta)^*(x_\beta)=x_\beta\cdot(1-x_\gamma)^{\langle\beta,\gamma\rangle}
\]
where $\gamma$ is the class of the unique horizontal saddle connection of~$\phi$ and we have used the fact that $\Omega_\phi(\gamma)=1$. Let us write $\Gamma_{\tau_j}$ as in Section~\ref{sec:TheMainResult} for the lattice spanned by tagged arcs of the tagged triangulation~$\tau_j$ determined by~$e^{-2i\theta_j}\cdot\phi$. The Gauss-Manin connection gives an isomorphism $\Gamma_{\tau_1}\stackrel{\sim}{\rightarrow}\Gamma_{\tau_2}$, and by Proposition~10.4 of~\cite{BridgelandSmith}, the induced isomorphism $\mathbb{T}_{\tau_1}\rightarrow\mathbb{T}_{\tau_2}$ is precisely the isomorphism~$\iota_k$ of Proposition~\ref{prop:clustertransformation} where $k$ is the edge being flipped. The transformation relating the maps~\eqref{eqn:sametarget} is therefore the map $\kappa_k$ of Proposition~\ref{prop:clustertransformation}. If we identify $\mathbb{T}_-$ with~$\mathbb{T}_+$ using the canonical basepoint~$\xi$, then this agrees with~$\mathbf{S}_\phi(\Delta)$ because $\xi(\gamma)=-1$.

If the unique horizontal saddle connection of $\phi\in F_2$ is closed, then the rotated differentials $e^{-2i\theta_j}\cdot\phi$ determine the same tagged triangulation, and therefore the transformation relating the maps~\eqref{eqn:sametarget} is the identity. Under our assumptions, the class $\gamma$ of the saddle connection satisfies $\Omega_\phi(\gamma)=0$ so that $\mathbf{S}_\phi(\Delta)$ is the identity as well. Hence $\phi$ is good.
\end{proof}

\begin{lemma}
\label{lem:Fp}
Every generic complete differential is good.
\end{lemma}

\begin{proof}
Since the union $\coprod_{p=0}^\infty F_p\subset\mathscr{Q}^\pm(\mathbb{S},\mathbb{M})$ is the set of all complete GMN differentials, we must show, for any $p\geq0$, that a generic differential in~$F_p$ is good. The case $p=1$ is vacuous because $F_1=\emptyset$. The cases $p=0$ and $p=2$ were handled in Lemmas~\ref{lem:F0} and~\ref{lem:F1}, respectively.

Let $p>2$ and assume inductively that all generic differentials in $F_{p-1}$ are good. Let $C$ denote a connected component of $F_p$. By Proposition~5.5 of~\cite{BridgelandSmith}, we can find, for any $\phi\in C$, an open neighborhood $\phi\in U_\phi\subset C$ and a constant $\varepsilon_\phi>0$ such that $e^{-2i\theta}\cdot q\in B_{p-1}$ when $0<|\theta|<\varepsilon_\phi$ and $q\in U_\phi$. By shrinking $U_\phi$ if necessary, we can assume further that there exist $\theta_j$ with $-\varepsilon_\phi<\theta_1<0<\theta_2<\varepsilon_\phi$ such that the differentials $e^{-2i\theta_j}\cdot q$ are saddle-free for all~$q\in U_\phi$. Let us write $\Delta\subset\mathbb{C}^*$ for the sector with boundary rays $\mathbb{R}_{>0}\cdot e^{i\theta_j}$. By shrinking $U_\phi$ once more if necessary, we can find a phase $s_\phi\in\mathbb{R}$ such that $\widetilde{q}=e^{-2is_\phi}\cdot q$ is saddle-free for all $q\in U_\phi$ and the rotated sector $\widetilde{\Delta}=e^{-is_\phi}\cdot\Delta$ lies in the upper half plane with non-active boundary rays for each~$\widetilde{q}$ with $q\in U_\phi$. If the differentials $q_t$, $t\in[0,1]$, form a path in~$U_\phi$ with generic endpoints, then we have $\mathbf{S}_{\widetilde{q}_0}(\widetilde{\Delta})=\mathbf{S}_{\widetilde{q}_1}(\widetilde{\Delta})$ by Theorem~\ref{thm:firstWCF}. From the definition of the BPS~automorphism given in Section~\ref{sec:BPSautomorphisms}, we get $\mathbf{S}_{q_0}(\Delta)=\mathbf{S}_{q_1}(\Delta)$. It then follows from the inductive hypothesis that $q_0$ is good if and only if $q_1$ is good. Since the sets $U_\phi$ form an open cover of~$C$ and the generic differentials are dense in~$F_p$, it will follow that all generic differentials in~$C$ are good if we can prove that one such differential is good.

Let $\phi$ be a point in the closure of~$C$ as in Proposition~\ref{prop:wallshaveends}. By Proposition~5.5 of~\cite{BridgelandSmith}, there exists a neighborhood $\phi\in U\subset B_p$ such that $F_p\cap U$ is contained in a hyperplane $D\subset\mathscr{Q}^\pm(\mathbb{S},\mathbb{M})$ given locally by a condition of the form $Z_\phi(\gamma)\in\mathbb{R}$ for some $\gamma\in\Gamma_\phi$. Thus there is a constant $\varepsilon>0$ such that $e^{-2i\theta}\cdot q\in B_{p-1}$ when $0<|\theta|<\varepsilon$ and $q\in U\cap D$. By shrinking this $U$ if necessary, we can assume there exist $\theta_j$ with $-\varepsilon<\theta_1<0<\theta_2<\varepsilon$ such that the differentials $e^{-2i\theta_j}\cdot q$ are saddle-free for all $q\in U\cap D$. Let $\Delta\subset\mathbb{C}^*$ be the sector with boundary rays $\mathbb{R}_{>0}\cdot e^{-2i\theta_j}$. By shrinking $U$ once more if necessary, we can find $s\in\mathbb{R}$ such that $\widetilde{q}=e^{-2is}\cdot q$ is saddle-free for all $q\in U\cap D$ and the sector $\widetilde{\Delta}=e^{-is}\cdot\Delta$ lies in the upper half plane with non-active boundary rays for each of these~$\widetilde{q}$. Then by Proposition~\ref{prop:wallshaveends}, there is a smaller neighborhood $\phi\in V\subset U$ for which $V\cap B_{p-1}$ is connected. Let $q_t$, $t\in[0,1]$, be a path in~$V\cap D$ connecting a generic~$q_0\in C$ to a generic differential~$q_1$ that lies off of~$F_p$. Applying Theorem~\ref{thm:firstWCF}, we see that $\mathbf{S}_{\widetilde{q}_0}(\widetilde{\Delta})=\mathbf{S}_{\widetilde{q}_1}(\widetilde{\Delta})$, hence $\mathbf{S}_{q_0}(\Delta)=\mathbf{S}_{q_1}(\Delta)$ as before. By the induction hypothesis, $q_1$ is good and therefore so is~$q_0$. It follows that any generic differential in~$C\cap V$ is good.
\end{proof}

Now suppose $\phi\in\mathscr{Q}^\pm(\mathbb{S},\mathbb{M})$ is any complete and generic differential. Note that we have  
\[
\mathbf{S}_\phi(\Delta_1\cup\Delta_2)=\mathbf{S}_\phi(\Delta_1)\circ\mathbf{S}_\phi(\Delta_2)
\]
whenever $\Delta_1$,~$\Delta_2\subset\mathbb{C}^*$ are adjacent sectors in the clockwise order whose boundary rays are non-active and whose union is a convex sector. We can use this fact to prove Theorem~\ref{thm:main} for a general convex sector $\Delta\subset\mathbb{C}^*$. Indeed, let $t\in[\theta_1,\theta_2]$ where $\theta_1$ and $\theta_2$ are the phases of the boundary rays of~$\Delta$. The rotated differential $e^{-2it}\cdot\phi$ is good by Lemma~\ref{lem:Fp}. Let $\varepsilon_t>0$ be the constant associated to $e^{-2it}\cdot\phi$ in the definition of a good differential. Then the open intervals $(t-\varepsilon_t,t+\varepsilon_t)$ form a cover of $[\theta_1,\theta_2]$. As this interval is compact, we can find finitely many $t_1,\dots,t_k\in[\theta_1,\theta_2]$ in increasing order such that the associated open intervals form a finite subcover. Choose $s_j\in[\theta_1,\theta_2]$ so that 
\[
\theta_1=s_1<t_1<s_2<t_2<\dots<t_k<s_{k+1}=\theta_2
\]
and so the differentials $e^{-2is_j}\cdot\phi$ are saddle-free and $t_j-s_j$,~$s_{j+1}-t_j<\varepsilon_{t_j}$ for $j=1,\dots,k$. For each~$j$, let $\Delta_j$ be the sector with boundary rays $\mathbb{R}_{>0}\cdot e^{is_j}$ and $\mathbb{R}_{>0}\cdot e^{is_{j+1}}$. Part~(2) of Theorem~\ref{thm:main} now follows since 
\[
\mathbf{S}_\phi(\Delta)=\mathbf{S}_\phi(\Delta_k)\circ\dots\circ\mathbf{S}_\phi(\Delta_1)
\]
and each factor $\mathbf{S}_\phi(\Delta_j)$ on the right is identified with the birational transformation of~$\mathbb{T}_+$ relating the Fock-Goncharov coordinates associated to $e^{-2is_j}\cdot\phi$ and $e^{-2is_{j+1}}\cdot\phi$.

\appendix
\section{The motivic wall-crossing formula}
\label{sec:TheMotivicWallCrossingFormula}

In this appendix, we explain how the version of the wall-crossing formula considered in Section~\ref{sec:TheWallCrossingFormulaForQuadraticDifferentials} is derived from the usual version in motivic Donaldson-Thomas theory. Since we are mainly interested in the application to quadratic differentials, our discussion will be quite cursory. For other introductory accounts of this material, see the references~\cite{KontsevichSoibelman2, Reineke, Keller, Meinhardt, Bridgeland18}.

\subsection{Quivers with potential}

Recall that a \emph{quiver} is simply a directed graph. It consists of a finite set $Q_0$ (the set of \emph{vertices}), a finite set $Q_1$ (the set of \emph{arrows}), and maps $s:Q_1\rightarrow Q_0$ and $t:Q_1\rightarrow Q_0$ taking an arrow to its \emph{source} and \emph{target}, respectively. We often write $a:i\rightarrow j$ to mean that $a$ is an arrow with $s(a)=i$ and $t(a)=j$.

Let $Q=(Q_0,Q_1,s,t)$ be a quiver. Then a complex \emph{representation} $M$ of~$Q$ consists of a complex vector space $M_i$ for each vertex $i\in Q_0$ and a linear map $M_a:M_i\rightarrow M_j$ for each arrow $a:i\rightarrow j$ in~$Q_1$. A representation is finite-dimensional if each of the vector spaces $M_i$ is finite-dimensional. In this case, we can define the \emph{dimension vector} to be the vector $\underline{\dim}(M)=(\dim_\mathbb{C}M_i)_{i\in Q_0}$. The finite-dimensional representations of~$Q$ form an abelian category where a morphism $M\rightarrow N$ of representations is a collection of linear maps $\eta_i:M_i\rightarrow N_i$ for $i\in Q_0$ satisfying the identity $\eta_{t(a)}\circ M_a=N_a\circ\eta_{s(a)}$ for every $a\in Q_1$.

Representations of a quiver can be viewed alternatively as modules over a certain noncommutative algebra spanned by paths in the quiver. By a \emph{path} in~$Q$, we mean a sequence of arrows $a_0,\dots,a_k$ such that $t(a_i)=s(a_{i-1})$ for $i=1,\dots,k$. Let us denote this path by $p=a_0\dots a_k$. We define its \emph{source} by $s(p)=s(a_k)$ and its \emph{target} by $t(p)=t(a_0)$. A path $p$ in~$Q$ is \emph{cyclic} if its source and target coincide. Two paths $p$ and $q$ are \emph{composable} if $s(p)=t(q)$, and in this case their composition $pq$ is defined by juxtaposition. To any quiver $Q$, we associate a $\mathbb{C}$-algebra $\mathbb{C}Q$ called the \emph{path algebra} of~$Q$. It is spanned as a vector space by the set of all paths in~$Q$. The product of two paths is defined to be their composition if the paths are composable and zero otherwise. Extending this operation linearly gives the multiplication on~$\mathbb{C}Q$. The following fact is well known.

\begin{lemma}
\label{lem:equivalence}
The category of finite-dimensional representations of a quiver $Q$ is equivalent to the category of finite-dimensional left modules over the path algebra $\mathbb{C}Q$. Under this equivalence, a representation $M$ of~$Q$ maps to the obvious module over~$\mathbb{C}Q$ whose underlying vector space is the direct sum $\bigoplus_{i\in Q_0}M_i$.
\end{lemma}

It is often important to consider representations of a quiver $Q$ where the linear maps are required to satisfy certain relations. For us, these relations will always come from a \emph{potential}, which is defined as a $\mathbb{C}$-linear combination of cyclic paths in~$Q$. If $p=a_1\dots a_k$ is a cyclic path in~$Q$ and $a\in Q_1$ is an arrow, then we can take the \emph{cyclic derivative} 
\[
\partial_a(p)=\sum_{i:a_i=a}a_{i+1}\dots a_ka_1\dots a_{i-1}\in\mathbb{C}Q.
\]
Extending this operation linearly, we can define $\partial_a(W)$ for any potential~$W$. Then the \emph{Jacobian ideal} is the ideal $\mathfrak{a}\subset\mathbb{C}Q$ generated by all cyclic derivatives of~$W$, and the \emph{Jacobian algebra} is the quotient $J(Q,W)=\mathbb{C}Q/\mathfrak{a}$.

To relate quiver representations and quadratic differentials, we should consider a slightly modified setup. Namely, we should take the completion $\widehat{\mathbb{C}Q}$ of the path algebra $\mathbb{C}Q$ with respect to the ideal generated by the arrows. Then the \emph{complete Jacobian algebra} is defined as the quotient of $\widehat{\mathbb{C}Q}$ by the closure of the Jacobian ideal. As shown in Section~10 of~\cite{DWZ}, a finite-dimensional module over the complete Jacobian algebra is the same thing as a nilpotent module over the usual Jacobian algebra, that is, a module annihilated by sufficiently long paths in~$Q$.

\subsection{Stability conditions}

The concept of a stability condition on a triangulated category was introduced by Bridgeland in~\cite{Bridgeland07}. In this appendix, we will be concerned with a version of this concept for abelian categories of quiver representations. Let $\mathcal{A}=\mathcal{A}(Q,W)$ denote the category of nilpotent modules over the Jacobian algebra of a quiver with potential $(Q,W)$. Then a \emph{stability condition} on~$\mathcal{A}$ is a tuple $\zeta\in\mathbb{H}_+^{Q_0}$ of complex numbers in the domain 
\[
\mathbb{H}_+=\{z=re^{i\pi\theta}:\text{$r>0$ and $0<\theta\leq1$}\}\subset\mathbb{C}.
\]

Let $\Gamma=\mathbb{Z}^{Q_0}$ be the lattice spanned by vertices of the quiver~$Q$. Given a stability condition~$\zeta$ on~$\mathcal{A}$, we define a group homomorphism $Z_\zeta:\Gamma\rightarrow\mathbb{C}$ called the \emph{central charge} by the rule
\[
Z_\zeta(\gamma)=\sum_{i\in Q_0}\gamma_i\cdot\zeta_i\in\mathbb{C}.
\]
We can then define the \emph{slope} of a nonzero vector $\gamma\in\mathbb{N}^{Q_0}$ to be the real number $\frac{1}{\pi}\cdot\arg Z_\zeta(\gamma)\in(0,1]$. The slope~$\mu(M)$ of an object $M\in\mathcal{A}$ is defined to be the slope of its dimension vector. A nonzero object $M\in\mathcal{A}$ is said to be \emph{semistable} if we have 
\[
\mu(N)\leq\mu(M)
\]
for every proper nonzero submodule $N\subset M$.

For convenience, we define, for any real number $\mu\in(0,1]$, the subset 
\[
\Lambda_\mu^\zeta=\{\gamma\in\mathbb{N}^{Q_0}:\text{$\gamma$ has slope $\mu$}\}\cup\{0\}\subset\Gamma.
\]
We denote by $\langle-,-\rangle$ the skew-symmetrized \emph{Euler pairing} on the lattice $\Gamma$ given by $\langle\alpha,\beta\rangle=(\alpha,\beta)-(\beta,\alpha)$ where  
\[
(\alpha,\beta)=\sum_i\alpha_i\beta_i-\sum_{a:i\rightarrow j}\alpha_i\beta_j.
\]
We say that a stability condition $\zeta$ is \emph{generic} if, for each $\mu$, one has $\langle \alpha,\beta\rangle=0$ for $\alpha$,~$\beta\in\Lambda_\mu^\zeta$.

\subsection{Motivic theories}

We now describe a general framework in which we can formulate the Kontsevich-Soibelman wall-crossing formula. Following~\cite{Meinhardt}, we define a \emph{motivic theory} to be a rule associating to a scheme~$X$ an abelian group~$R(X)$. This rule is required to be functorial in two ways: If $u:X\rightarrow Y$ is any morphism of schemes, there is a homomorphism $u^*:R(Y)\rightarrow R(X)$ called the \emph{pullback} along~$u$, and if $u$ is of finite type, there is a homomorphism $u_!:R(X)\rightarrow R(Y)$ called the \emph{pushforward} along~$u$. In addition, there is an associative, symmetric operation 
\[
\boxtimes:R(X)\otimes_{\mathbb{Z}}R(Y)\rightarrow R(X\times Y)
\]
called the \emph{exterior product} with unit element $1\in R(\Spec\mathbb{C})$, and operations 
\[
\sigma^n:R(X)\rightarrow R(\Sym^n(X))
\]
for each~$n\in\mathbb{N}$ where $\Sym(X)=X^n\sslash S_n$ is the $n$th symmetric product. These data are required to satisfy several axioms listed in Section~4 of~\cite{Meinhardt}.

\begin{example}
A fundamental example of a motivic theory is the theory $R=\underline{\mathrm{K}}_0(\Sch)$, which associates to a connected scheme~$X$ the group $R(X)=\mathrm{K}_0(\Sch_X)$ generated by isomorphism classes $[V\rightarrow X]$ of schemes $V$ of finite type over~$X$ subject to the relation 
\[
[V\rightarrow X]=[Z\rightarrow X]+[V\setminus Z\rightarrow X]
\]
for any closed subscheme $Z\subset V$. We extend this to nonconnected schemes $X$ by defining $R(X)=\prod_{X_i\in\pi_0(X)}K_0(\Sch_{X_i})$. To get the structure of a motivic theory, we define the pullback by 
\[
u^*\left([W\rightarrow Y]\right)=[X\times_Y W\rightarrow X]
\]
for any morphism $u:X\rightarrow Y$ and define the pushforward by 
\[
u_!\left([V\stackrel{v}{\rightarrow}X]\right)=[V\stackrel{u\circ v}{\longrightarrow}Y]
\]
for any morphism $u:X\rightarrow Y$ of finite type. The exterior product is given by 
\[
[V\stackrel{v}{\rightarrow}X]\boxtimes[W\stackrel{w}{\rightarrow}Y]=[V\times W\stackrel{v\times w}{\longrightarrow}X\times Y]
\]
with unit element $1=[\Spec\mathbb{C}\stackrel{\mathrm{id}}{\rightarrow}\Spec\mathbb{C}]\in R(\Spec\mathbb{C})$, and the $\sigma^n$ operations are given by 
\[
\sigma^n\left([V\rightarrow X]\right)=[\Sym^n(V)\rightarrow\Sym^n(X)]
\]
for each $n\in\mathbb{N}$.
\end{example}

Given any motivic theory~$R$, we can define a product operation on~$R(\Spec\mathbb{C})$ by the rule $ab=+_!(a\boxtimes b)$ where $+:\Spec\mathbb{C}\times\Spec\mathbb{C}\rightarrow\Spec\mathbb{C}$ is the obvious isomorphism. In this way, the abelian group $R(\Spec\mathbb{C})$ becomes a ring with unit element~1. We can also associate to each scheme $X$, the element $[X]_R\coloneqq c_!c^*(1)\in R(\Spec\mathbb{C})$ where $c:X\rightarrow\Spec\mathbb{C}$ is the constant map. In particular, we define the element $\mathbb{L}_R\coloneqq[\mathbb{A}^1]_R$.

\begin{lemma}[\cite{Meinhardt}, Exercise~4.6]
The identity 
\[
[\mathrm{GL}(n)]_R=\prod_{i=0}^{n-1}(\mathbb{L}_R^n-\mathbb{L}_R^i)
\]
holds in the ring $R(\Spec\mathbb{C})$.
\end{lemma}

The notion of a motivic theory can be extended to the notion of a \emph{stacky motivic theory}. This is an operation $\mathfrak{R}$ that assigns to every stack of the form $\mathfrak{X}=\coprod_iX_i/G_i$, where the $X_i$ are schemes and the $G_i$ are linear algebraic groups, an abelian group $\mathfrak{R}(\mathfrak{X})$. One has a pullback $u^*:\mathfrak{R}(\mathfrak{Y})\rightarrow\mathfrak{R}(\mathfrak{X})$ for any 1-morphism $u:\mathfrak{X}\rightarrow\mathfrak{Y}$, and a pushforward $u_!:\mathfrak{R}(\mathfrak{X})\rightarrow\mathfrak{R}(\mathfrak{Y})$ if $u$ is of finite type in the sense of~\cite{Meinhardt}. One also has structures $\boxtimes$, $1\in\mathfrak{R}(\Spec\mathbb{C})$, and $\sigma^n$ for $n\in\mathbb{N}$ satisfying various properties as in the definition of a motivic theory for schemes.

Let $R$ be any motivic theory satisfying the condition $\sigma^n(a\mathbb{L}_R)=\sigma^n(a)\mathbb{L}_R^n$ for any $a\in R(X)$ with~$X$ any scheme and $n\in\mathbb{N}$. As explained in~\cite{Meinhardt}, Section~4.3, there is a functorial construction which associates to~$R$ a stacky motivic theory $R^\mathrm{st}$ whose value on a connected scheme~$X$ is the group 
\[
R^\mathrm{st}(X)\coloneqq R(X)\left[[GL(n)]_R^{-1}:n\in\mathbb{N}\right]
\]
defined using the $R(\Spec\mathbb{C})$-module structure on~$R(X)$ induced by the exterior product.

\subsection{Motivic DT invariants}

Let us fix a motivic theory $R$ satisfying $\sigma^n(a\mathbb{L}_R)=\sigma^n(a)\mathbb{L}_R^n$ for any $a\in R(X)$ with~$X$ any scheme and $n\in\mathbb{N}$. As in~\cite{Meinhardt}, Example~6.1, we can apply the results of~\cite{DavisonMeinhardt1}, Appendix~B, to extend the $\sigma^n$ operations to $R(X)[\mathbb{L}_R^{1/2}]$ to get a new motivic theory and then pass to the associated stacky motivic theory $\mathfrak{R}=R[\mathbb{L}_R^{1/2}]^\mathrm{st}$. The constructions described below will depend on a choice of stacky vanishing cycle with values in~$\mathfrak{R}$, which can be chosen canonically (see~\cite{Meinhardt}, Section~5 for the definition and construction of vanishing cycles).

Suppose we are given a stability condition $\zeta$ on~$\mathcal{A}=\mathcal{A}(Q,W)$ and a real number $\mu\in(0,1]$. Let us write $\mathcal{S}(\mathcal{A})$ for the abelian group $\mathfrak{R}(\mathbb{N}^{Q_0})\cong\prod_{\gamma\in\mathbb{N}^{Q_0}}\mathfrak{R}(\Spec\mathbb{C})\cong\mathfrak{R}(\Spec\mathbb{C})\llbracket x_\gamma:\gamma\in\mathbb{N}^{Q_0}\rrbracket$. Then we can construct an element 
\[
\mathcal{DT}(\mathcal{A})_\mu^\zeta\in\mathcal{S}(\mathcal{A})
\]
called the \emph{Donaldson-Thomas function}. The precise definition of this object is rather technical, and therefore we will only sketch the construction here, referring to~\cite{Meinhardt} for further details. To begin, we equip the abelian group $\mathcal{S}(\mathcal{A})\cong\mathfrak{R}(\Spec\mathbb{C})\llbracket x_\gamma:\gamma\in\mathbb{N}^{Q_0}\rrbracket$ with an auxiliary product $*$. It is the unique $\mathfrak{R}(\Spec\mathbb{C})$-bilinear continuous product defined on generators by 
\[
x_\alpha*x_\beta=\mathbb{L}_R^{\langle\alpha,\beta\rangle/2}\cdot x_{\alpha+\beta}.
\]
Associated to the category~$\mathcal{A}$, there is an associative algebra $\mathcal{H}(\mathcal{A})$ known as the \emph{Ringel-Hall algebra}, and there is an algebra homomorphism 
\[
\mathcal{I}:\mathcal{H}(\mathcal{A})\rightarrow\mathcal{S}(\mathcal{A})
\]
known as the \emph{integration map}. (To be somewhat more precise, the map $\mathcal{I}$ that we consider here is the composition of the integration map defined in~\cite{Meinhardt} with the pushforward along the map $\underline{\dim}$ sending a $\mathbb{C}Q$-module to its dimension vector.) As explained in Section~6.3 of~\cite{Meinhardt}, there is an element $\delta_\mu^\zeta\in\mathcal{H}(\mathcal{A})$ naturally associated to each slope $\mu\in(0,1]$. We define $\mathcal{DT}(\mathcal{A})_\mu^\zeta$ to be the unique element of $\mathcal{S}(\mathcal{A})\cong\mathfrak{R}(\Spec\mathbb{C})\llbracket x_\gamma:\gamma\in\mathbb{N}^{Q_0}\rrbracket$ with constant term zero such that 
\[
\mathcal{I}(\delta_\mu^\zeta)=\Sym\left(\mathcal{DT}(\mathcal{A})_\mu^\zeta\right)
\]
where $\Sym(a)\coloneqq\sum_{n\in\mathbb{N}}\Sym^n(a)$. It follows from Lemma~6.3 of~\cite{Meinhardt} that such an element $\mathcal{DT}(\mathcal{A})_\mu^\zeta$ is well defined.

\subsection{The wall-crossing formula}

We can now formulate a version of the wall-crossing formula for motivic DT~invariants.

\begin{theorem}[\cite{DavisonMeinhardt1}, Proposition~6.23 and~\cite{Meinhardt}, Section~6.5]
\label{thm:motivicWCF}
Let $D\subset(0,1]$ be an interval. Then the product 
\[
\prod_{\mu\in D}\mathcal{I}(\delta_\mu^\zeta)\in\mathcal{S}(\mathcal{A})
\]
is constant as the stability condition $\zeta$ varies, provided there is no semistable object in~$\mathcal{A}$ whose slope enters or leaves the interval~$D$. Here the product is taken with respect to the multiplication $*$ in order of decreasing slopes.
\end{theorem}

The infinite product appearing in Theorem~\ref{thm:motivicWCF} is well defined because for every $\gamma\in\mathbb{N}^{Q_0}$, only finitely many factors contribute to its $\gamma$-component in $\mathcal{S}(\mathcal{A})\cong\prod_{\gamma\in\mathbb{N}^{Q_0}}\mathfrak{R}(\Spec\mathbb{C})$.

Some remarks are in order as the statement of Theorem~\ref{thm:motivicWCF} differs slightly from the wall-crossing formula proved in~\cite{DavisonMeinhardt1} and reviewed in~\cite{Meinhardt}. First of all, these references consider arbitrary modules over the Jacobian algebra whereas we consider only nilpotent modules. In fact, the proof of the wall-crossing formula for nilpotent modules is essentially the same; one simply replaces all moduli stacks of quiver representations by the corresponding stacks of nilpotent representations.

Another difference between Theorem~\ref{thm:motivicWCF} and the wall-crossing formulas formulated in~\cite{DavisonMeinhardt1} and~\cite{Meinhardt} is that we take the product over slopes in an interval $D\subset(0,1]$ whereas \cite{DavisonMeinhardt1} and~\cite{Meinhardt} take the product over all slopes. Again, the idea of the proof is the same. In the proof of Theorem~\ref{thm:motivicWCF}, one considers objects of~$\mathcal{A}$ whose Harder-Narasimhan factors have slopes contained in the interval~$D$. (See~\cite{Meinhardt} for details on the Harder-Narasimhan filtration and the proof of the wall-crossing formula.)

\subsection{Quasi-classical limit}

As explained in~\cite{Meinhardt}, Example~5.7, there is a motivic theory $R$ that associates to a variety~$X$ the Grothendieck group of the category $D^b(\mathrm{MMHM}(X))$ of ``monodromic mixed Hodge modules'' on~$X$. We refer to~\cite{DavisonMeinhardt2} for a more detailed description of this category. We can upgrade this~$R$ to a stacky motivic theory~$\mathfrak{R}$ as above and view the wall-crossing formula as an identity in the group~$\mathcal{S}(\mathcal{A})=\mathfrak{R}(\mathbb{N}^{Q_0})$.

Let us write $\mathcal{T}(\mathcal{A})=\mathbb{Z}(\!(q^{1/2})\!)\llbracket x_i:i\in Q_0\rrbracket$. In Section~1.3 of~\cite{DavisonMeinhardt2}, the authors associate to each object $\mathcal{L}\in D^b(\mathrm{MMHM}(\mathbb{N}^{Q_0}))$, a formal power series 
\[
\mathcal{Z}(\mathcal{L})=\sum_{\gamma\in\mathbb{N}^{Q_0}}\chi(\mathcal{L}_\gamma,q^{1/2})\cdot x_\gamma \in \mathcal{T}(\mathcal{A})
\]
where for each $\gamma\in\mathbb{N}^{Q_0}$ we have $\chi(\mathcal{L}_\gamma,q^{1/2})\in\mathbb{Z}(\!(q^{1/2})\!)$ and $x_\gamma\coloneqq\prod_ix_i^{\gamma_i}$. This operation $\mathcal{Z}$ provides a map from a subgroup of~$\mathcal{S}(\mathcal{A})$ to~$\mathcal{T}(\mathcal{A})$. In the terminology of~\cite{DavisonMeinhardt2}, the category $\mathcal{A}=\mathcal{A}(Q,W)$ of nilpotent modules over the Jacobian algebra of the quiver with potential $(Q,W)$ is a Serre subcategory of the category of all finite-dimensional modules over the Jacobian algebra. The integrality theorem proved in~\cite{DavisonMeinhardt2} therefore implies that for a generic stability condition $\zeta$ on~$\mathcal{A}$, one has an identity 
\[
\mathcal{Z}(\mathcal{I}(\delta_\mu^\zeta))=\operatorname{EXP}\left(\sum_{0\neq\gamma\in\Lambda_\mu^\zeta} \frac{\Omega_\gamma^\zeta(q^{-1/2})}{q^{1/2}-q^{-1/2}}\cdot x_\gamma\right).
\]
Here $\Omega_\gamma^\zeta(q^{1/2})$ is a Laurent polynomial in~$q^{1/2}$ called the \emph{refined BPS invariant} for $\gamma$ and~$\zeta$, and $\operatorname{EXP}$ denotes the \emph{plethystic exponential}. The latter is a homomorphism from the additive group $\mathfrak{m}\subset\mathbb{Z}(q^{1/2})\llbracket x_i:i\in Q_0\rrbracket$ of series with constant term zero to the multiplicative group $1+\mathfrak{m}$ defined by the rule 
\[
\operatorname{EXP}\left(f(q^{1/2}\right)\cdot x_\gamma)=\exp\left(\sum_{n=1}^\infty\frac{1}{n}f(q^{n/2})\cdot x_{n\gamma}\right)
\]
for any $f(q^{1/2})\in\mathbb{Z}(q^{1/2})$ and $\gamma\in\mathbb{N}^{Q_0}$.

Just as we defined an auxiliary product $*$ on~$\mathcal{S}(\mathcal{A})$, we can define an auxiliary product, also denoted~$*$, on the ring $\mathcal{T}(\mathcal{A})$. It is the unique $\mathbb{Z}(\!(q^{1/2})\!)$-bilinear continuous product defined on basis elements by 
\[
x_\alpha*x_\beta=(-q^{1/2})^{\langle\alpha,\beta\rangle}\cdot x_{\alpha+\beta}.
\]
Then map $\mathcal{Z}$ preserves the $*$~products. Note that by skew-symmetry of the pairing $\langle-,-\rangle$, one has $x_\alpha*x_\beta=q^{\langle\alpha,\beta\rangle}\cdot x_\beta*x_\alpha$. There is an automorphism $\mathbf{S}_\mu^\zeta$ of the algebra $\mathcal{T}(\mathcal{A})$ with respect to the product~$*$ given by conjugation with $\mathcal{Z}(\mathcal{I}(\delta_\mu^\zeta))$:
\[
\mathbf{S}_\mu^\zeta(a)=\mathcal{Z}(\mathcal{I}(\delta_\mu^\zeta))*a*\mathcal{Z}(\mathcal{I}(\delta_\mu^\zeta))^{-1}.
\]
For a generic stability condition $\zeta$, it follows from the definition of the plethystic exponential and the commutation relation for~$*$ that this automorphism acts on basis elements by 
\[
\mathbf{S}_\mu^\zeta(x_\beta)=x_\beta*\operatorname{EXP}\left(-\sum_{0\neq\gamma\in\Lambda_\mu^\zeta}q^{-1/2}[\langle\beta,\gamma\rangle]_{q^{-1}}\,\Omega_\gamma^\zeta(q^{-1/2})\cdot x_\gamma\right)
\]
where we have introduced the \emph{quantum integer} $[n]_t\coloneqq\frac{t^n-1}{t-1}=1+t+\dots+t^{n-1}$. Using the definition, one can compute the plethystic exponential explicitly, and in the quasi-classical limit $q^{1/2}\rightarrow1$, one finds 
\[
\mathbf{S}_\mu^\zeta(x_\beta)=x_\beta*\prod_{0\neq\gamma\in\Lambda_\mu^\zeta}(1-x_\gamma)^{\Omega_\zeta(\gamma)\cdot\langle\beta,\gamma\rangle}
\]
where we have written $\Omega_\zeta(\gamma)=\Omega_\gamma^\zeta(1)$. Taking this quasi-classical limit in Theorem~\ref{thm:motivicWCF}, we obtain the following version of the wall-crossing formula.

\begin{theorem}
\label{thm:numericalWCF}
Let $D\subset(0,1]$ be an interval. Suppose $\zeta_t$, $t\in[0,1]$, is a path in the space of stability conditions with generic endpoints and there is no semistable object with slope in the boundary of $D$ for any $\zeta_t$. Then we have 
\[
\prod_{\mu\in D}\mathbf{S}_\mu^{\zeta_0}=\prod_{\mu\in D}\mathbf{S}_\mu^{\zeta_1}
\]
in $\Aut\mathbb{Z}\llbracket x_i:i\in Q_0\rrbracket$. Here the products are taken in order of decreasing slopes.
\end{theorem}

\subsection{Link with quadratic differentials}

As explained in~\cite{BridgelandSmith}, a tagged triangulation $\tau$ of a marked bordered surface $(\mathbb{S},\mathbb{M})$ determines an associated quiver with potential $(Q(\tau),W(\tau))$, and the domain $\mathcal{C}_\tau\subset\mathscr{Q}^\pm(\mathbb{S},\mathbb{M})$ in the moduli space of quadratic differentials embeds naturally into the space of stability conditions on $\mathcal{A}=\mathcal{A}(Q(\tau),W(\tau))$. If $\phi\in\mathcal{C}_\tau$ is a quadratic differential and $\zeta$ is the corresponding stability condition on~$\mathcal{A}$, then there is an isomorphism $\Gamma_\phi\cong\mathbb{Z}^{Q(\tau)_0}$, and the period map~$Z_\phi$ coincides with the central charge~$Z_\zeta$ under this isomorphism. In particular, if $\phi$ is generic then so is $\zeta$. By Theorem~1.4 of~\cite{BridgelandSmith}, the invariant $\Omega_\phi(\gamma)$ coincides with $\Omega_\zeta(\gamma)$ where we use the same notation for a class in $\Gamma_\phi$ and the corresponding element of $\mathbb{Z}^{Q(\tau)_0}$. If the boundary rays of a convex sector~$\Delta\subset\mathbb{C}^*$ are non-active for the differential~$\phi$ and $D\subset(0,1]$ is the set of phases of rays in~$\Delta$, then one can show as in Lemma~11.4 of~\cite{BridgelandSmith} that there is no semistable object for $\zeta$ with slope in the boundary of~$D$. Hence Theorem~\ref{thm:firstWCF} follows from Proposition~\ref{prop:KSautomorphism} and Theorem~\ref{thm:numericalWCF}.

\bibliographystyle{amsplain}

\end{document}